\UseRawInputEncoding
\documentclass[11pt]{article}

\usepackage{amsmath,latexsym,amsthm,verbatim,amssymb,amsfonts,amscd,graphicx,graphics,mathtools,nccmath,braket,tipa,braket,tikz-cd,mathdots,nicematrix,setspace,dsfont,xcolor,url}
\usepackage[shortlabels]{enumitem}
\usepackage[affil-it]{authblk}
\makeatletter
\g@addto@macro\th@plain{\thm@headpunct{}}
\makeatother

\theoremstyle{plain}
\newtheorem{theorem}{Theorem}[section]
\newtheorem{corollary}[theorem]{Corollary}

\newtheorem{definition}[theorem]{Definition}

\DeclarePairedDelimiter\floor{\lfloor}{\rfloor}

\DeclarePairedDelimiter\abs{\lvert}{\rvert}
\DeclarePairedDelimiter\norm{\lVert}{\rVert}

\makeatletter
\let\oldabs\abs
\def\abs{\@ifstar{\oldabs}{\oldabs*}}

\let\oldnorm\norm
\def\norm{\@ifstar{\oldnorm}{\oldnorm*}}
\makeatother

\newcommand\permsim{\mathrel{\overset{\makebox[0pt]{\mbox{\normalfont\tiny\sffamily perm}}}{\simeq}}}
\newcommand\unitsim{\mathrel{\overset{\makebox[0pt]{\mbox{\normalfont\tiny\sffamily unitary}}}{\simeq}}}

\makeatletter
\newcommand*\bigcdot{\mathpalette\bigcdot@{.5}}
\newcommand*\bigcdot@[2]{\mathbin{\vcenter{\hbox{\scalebox{#2}{$\m@th#1\bullet$}}}}}
\makeatother

\providecommand{\keywords}[1]{\textbf{\textit{Keywords:}} #1}
\providecommand{\subjclass}[1]{\textbf{\textit{MSC Codes:}} #1}

\begin{document}
 
\title{Antidiagonal Operators, Antidiagonalization, Hollow Quasidiagonalization -- Unitary, Orthogonal, Permutation, and Otherwise -- and Symmetric Spectra}
\author{David Nicholus\thanks{drnicholus@gmail.com}}

\affil{Chicago Transit Authority\\
  William Rainey Harper College, Department of Mathematics}
  
\maketitle

\begin{abstract}
After summarizing characteristics of antidiagonal operators, we derive three direct sum decompositions characterizing antidiagonalizable linear operators -- the first up to permutation-similarity, the second up to similarity, and the third up to unitary similarity. Each corresponds to a unique quasidiagonalization. We prove the permutation-similarity direct sum decomposition defines a hollow quasidiagonalization of a traceless antidiagonalizable operator and gives the real Schur decomposition of a real antisymmetric antidiagonal operator. We use this to derive an orthogonal antidiagonalization of a general real antisymmetric operator. We prove the similarity direct sum decomposition defines the eigendecomposition of an antidiagonalizable operator that is diagonalizable, and we give a characterization of this eigendecomposition. We show it also defines the Jordan canonical form for a general antidiagonalizable operator. This leads to a further characterization of antidiagonalizable operators in terms of spectral properties and a characterization as the direct sum of traceless $2 \times 2$ matrices with the exception of a single $1 \times 1$ matrix as an additional summand for operators of odd size. We discuss numerous implications of this for properties of the square of an antidiagonalizable operator, a characterization of operators that are both diagonalizable and antidiagonalizable, nilpotency of antidiagonalizable operators, unitary diagonalizations of normal antidiagonal operators, symmetric and antisymmetric antidiagonalizations, and centrosymmetric diagonalizations and antidiagonalizations. Finally, we prove the unitary similarity direct sum decomposition defines the Schur decomposition, as well as a unitary quasidiagonalization, of a unitarily antidiagonalizable operator.
\end{abstract}

\\


%\address{Computer Science Department\\
%         University of Winnebago\\
 %        Winnebago, MN 53714} 
%\email{drnicholus@gmail.com}
%\urladdr{http://math.uwinnebago.edu/homepages/menuhin/}
%\subjclass[2010]{Primary: 06B10; Secondary: 06D05}
%\thanks{Research supported by the NSF under grant number 23466.} 

\section{Introduction} \label{Sect:intro} 

An \emph{antidiagonal matrix} (often called a skew-diagonal matrix) is a matrix whose only nonzero elements lie on its antidiagonal. Antdiagonal operators, and more prominently, their various similarity classes, find a wide range of applications.

Antidiagonal binary matrices measure the degree to which an input sequence (the test sequence) is a palindrome of a given sequence (the database sequence). The sum of the antidiagonal elements in the similarity matrix\footnote{Similarity in this context is not referring to matrix-theoretic similarity; it is terminology referring to the alikeness of two sequences.} (where one sequence is the header row and the other is the header column) indicates in how many places the database sequence agrees with the reverse of the test sequence. For example, this sum would be the complement of the Hamming distance between the database sequence and the reverse of the test sequence. Such a matrix can also be formalized as the adjacency matrix of a vertex-ordered digraph. If the alphabet from which the elements are drawn is equipped with a metric indicating a distance between characters, the corresponding \emph{real} antidiagonal matrix measures how close the test sequence is to being a palindrome of the database sequence with respect to the metric. This describes the antidiagonal of a similarity matrix from an application of a pattern-matching search such as one given by the Smith-Waterman (SW) algorithm. \cite{aK11} Pattern-matching applications are countless, but a notable example is determining alignment between RNA sequences. \cite{aK11} The SW algorithm can also be parallelized by operating along the antidiagonal. \cite{aK11}

Adjacency matrices of bipartite graphs are block antidiagonal matrices. Matrix functions of block antidiagonal matrices are used in solving matrix differential equations, for computing block exponential-dependent functions, and in control theory, as discussed in \cite{aS16}. Simplifying and computing common matrix functions of block antidiagonal matrices are also discussed in \cite{aS16}.

We will provide no citation, but we believe many mathematicians have, at some point in their mathematical career, wondered about the properties, use, and importance of antidiagonal matrices as well as some form of ``antidiagonalization'' and ``unitary antidiagonalization'' given the immense importance placed on diagonal matrices, diagonalization, and unitary diagonalization in mathematics. Part of the motivation of this paper is to satisfy that curiosity, as we have felt it as well. We also discuss pedagogical value to antidiagonal matrices and their various similarity classes in Section \ref{Subsect:jordan}.

An \emph{antidiagonalizable} matrix is a matrix that is similar to an antidiagonal matrix. A matrix has a \emph{symmetric spectrum} if and only if its spectrum remains invariant under multiplication by $-1$. Antidiagonalizable matrices play a significant role in combinatorics and graph theory. There has been a recent surge of interest in graphs whose adjacency matrices have a symmetric spectrum, with especial focus on signed graphs\footnote{A \emph{signed graph} is a graph where every edge is assigned a positive or negative value.}. \cite{wH23,fB18,eG20,gG22,wH22,zS21,fR19,pW23,sA18,eR04} Among many examples, adjacency matrices of signed bipartite graphs have symmetric spectra and so do the adjacency matrices of pairs of cospectral signed graphs where one is bipartite and the other is not. \cite{wH23} Adjacency matrices of undirected graphs (including simple graphs, multigraphs, signed graphs, and weighted graphs) are diagonalizable due to the fact that such matrices are real symmetric matrices, so are normal by the spectral theorem for normal matrices. We prove any diagonalizable matrix with a symmetric spectrum is antidiagonalizable. Thus, undirected graphs whose adjacency matrices have a symmetric spectrum are antidiagonalizable. In fact, an unsigned graph is bipartite if and only if its adjacency matrix has a symmetric spectrum.\footnote{This is false for signed graphs. \cite{wH23}} \cite{wH23} In light of the preceding discussion, this implies an unsigned graph is bipartite if and only if it is antidiagonalizable. Similarly, adjacency matrices of sign-symmetric graphs have symmetric spectra.\footnote{Some Seidel matrices are counterexamples to the converse. \cite{fB18}} \cite{fB18} Thus, adjacency matrices of sign-symmetric graphs are antidiagonalizable. The adjacency matrix of a digraph need not be diagonalizable, however a large class of digraphs have diagonalizable adjacency matrices. \cite{yL22} If such matrices have a symmetric spectrum, they are also antidiagonalizable.

There has been recent interest in conference matrices with symmetric spectra. \cite{wH22} In particular, if a conference matrix is symmetric, then it has a symmetric spectrum. \cite{wH22}

A square matrix whose main diagonal consists only of 0s is a \emph{hollow matrix}, and a similarity decomposition bringing a matrix into hollow form is a \emph{hollowization}. Notice any graph that is not a multigraph has an adjacency matrix that is hollow. Notice also every traceless antidiagonal matrix is hollow. One of the most significant theorems pertaining to hollow matrices is the theorem due to Fillmore in \cite{pF69} asserting every traceless real square matrix is orthogonally similar to a hollow matrix. This stems from earlier research by Horn and Schur. \cite{aH54,iS23} There has been significant research centered around hollow matrices in recent years. \cite{zC13,mF15,hK16,aN18,tD20} Specific structural forms of hollow matrices, mostly symmetric, have been studied in \cite{zC13,mF15,hK16}. Hollowization, unitary hollowization, simultaneous hollowization, and simultaneous unitary hollowization have received attention very recently. \cite{aN18,tD20} In \cite{aN18}, Neven and Bastin use simultaneous unitary hollowization for separability in quantum mechanics and prove a certain quantum state is separable if and only if given symmetric matrices are simultaneously unitarily hollowizable. In \cite{tD20}, Damm and Fa\ss bender use simultaneous hollowization to prove theorems for stabilization of linear systems by rotational forces or by noise, give a constructive proof of Brickman's theorem from \cite{lB61} that the real joint numerical range of two matrices is convex, and prove the stronger version of Fillmore's theorem that every traceless real square matrix is orthogonal-symplectically similar to a hollow matrix. However, from our exploration of the literature, we agree with Damm and Fa\ss bender that (of, and in, their paper \cite{tD20}) ``to the best of our knowledge, the current note is the first to treat hollowization problems from the matrix theoretic side''. Our contribution is partially an exploration of a specific structural form of hollow matrices, in particular, general \emph{hollow quasidiagonal}\footnote{see Definition \ref{D:quasidiag}} matrices, as well as a matrix-theoretic treatment of hollowization to this form. Quasidiagonal operators have not received very much attention \cite{gL75}, so this paper contributes to a discussion of them.

Every hollow tridiagonal Toeplitz matrix has unique eigenvalues, so is diagonalizable, and has a symmetric spectrum, so is antidiagonalizable. More generally, every tridiagonal Toeplitz matrix, after subtracting a scalar matrix, results in a hollow tridiagonal Toeplitz matrix, which is antidiagonalizable. \cite{dK99,sN13} (Equivalently, every tridiagonal Toeplitz matrix has a spectrum that is a constant shift away from being symmetric.) Tridiagonal Toeplitz matrices play a prominent role in solid-state physics and quantum mechanics due to their widespread use in tight-binding models, the tridiagonal matrix equation-of-motion method, the calculation of electronic band structure, Harper's model, and studying the quantum Hall effect, which have received recent attention. \cite{rM17,jA04,hK13,rF00,rD69,mC80,mC76,mA17oct29,mA17oct19} Many commonly studied block tridiagonal tight-binding Hamiltonians are also antidiagonalizable. \cite{hK13,bN20} Further information and novel applications, including to inverse eigenvalue problems, Tikhonov regularization, and the construction of Chebyshev polynomial-based Krylov subspace bases, are elaborated in \cite{sN13}.

Every antisymmetric matrix is hollow. We prove in Section \ref{Subsect:realSchur} that every real antisymmetric matrix is orthogonally antidiagonalizable. Real antisymmetric matrices generate the special orthogonal Lie algebra $\mathfrak{so}(n, \mathbb{R})$ -- the tangent space to the orthogonal group $O(n, \mathbb{R})$ at the identity. Thus, antisymmetric matrices are generators of infinitesimal rotations. \cite{mA98} Consequently, they have well-known applications in physics. Moreover, there is recent interest in the application of real antisymmetric matrices to neural networks and machine learning \cite{pA20,kH22,kM20}, as well as to numerical analysis \cite{sS20}.

In Section \ref{Sect:prelim} we review definitions, terminology, and background information that we refer to throughout this paper, with especial attention given to different types of similarity.

In Section \ref{Sect:props} we discuss antidiagonal matrices and their properties. We give formulas for products, inverses, and powers of antidiagonal matrices and discuss arguably the most important antidiagonal matrices -- \emph{exchange matrices}\footnote{see Definition \ref{D:excMat}}. Finally, we provide the standard form for antidiagonal matrices we refer to throughout this paper.

In Section \ref{Sect:permDirectSum} we prove there is a very strong isomorphism between traceless antidiagonalizations and hollow quasidiagonalizations; traceless antidiagonalizations and hollow quasidiagonalizations are permutation-similar. More generally, we prove permutation-similarity between antidiagonalizations and \emph{Q-pseudo-hollow quasidiagonalizations} -- quasidiagonalizations where the only allowed nonzero diagonal element is also the only nonzero element in its row and column. Due to the strength of this isomorphism, most statements about antidiagonalizations can be converted to statements about hollow quasidiagonalizations, and vice versa; they describe essentially the same algebraic object. In this section, we also present the first of our three direct sum decompositions -- the direct-sum decomposition of an antidiagonalizable operator up to permutation-similarity. We then discuss some implications this decomposition has for real matrices in Section \ref{Subsect:realSchur}. In particular, we show the decomposition gives the real Schur decomposition\footnote{see Definition \ref{D:realSchurDef}} of a real antisymmetric antidiagonal matrix. We also show the decomposition provides an orthogonal antidiagonalization of a general real antisymmetric matrix.

We begin Section \ref{Sect:simDirectSum} by discussing the spectral properties of antidiagonalizable operators, including the spectrum, determinant, and trace. We prove every antidiagonalizable matrix of even size has a symmetric spectrum, and every antidiagonalizable matrix of odd size has a c-symmetric spectrum\footnote{see Definition \ref{D:csymSpec}}. In Section \ref{Subsect:eigendecomp} we derive the eigendecomposition of diagonalizable antidiagonalizable matrices and provide a characterization of it, and it effectively serves as a lemma to the Jordan canonical decomposition that follows in Section \ref{Subsect:jordan} as well as the second of our three direct sum decompositions -- the direct-sum decomposition of an antidiagonalizable operator up to similarity. This leads to a characterization of antidiagonalizable operators. In particular, we prove a liner operator $M$ is antidiagonalizable if and only if $M$ can be expressed as a direct sum of traceless $2 \times 2$ matrices, with the exception of a single $1 \times 1$ matrix as an additional summand for odd $n$. In terms of spectral properties, $M$ is antidiagonalizable if and only if $M$ has a symmetric spectrum for even $n$ and a c-symmetric spectrum for odd $n$, whereby the only generalized eigenvectors of rank $\neq 1$ are of rank 2 with eigenvalues of 0. The square of an antidiagonalizable matrix has some interesting properties mentioned in Theorem \ref{T:antiEigendec}, but Section \ref{Subsect:square} delves into more detail, where we prove the square of an antidiagonalizable matrix is diagonalizable and give conditions for when the square of an antidiagonalizable matrix is normal, Hermitian, positive semidefinite Hermitian, and negative semidefinite Hermitian. In Section \ref{Subsect:duodiag} we discuss \emph{duodiagonalizable} matrices -- matrices that are both diagonalizable and antidiagonalizable. We first prove if an antidiagonalizable matrix is nonsingular, then it is diagonalizable. Then we show a diagonalizable matrix $M$ is antidiagonalizable if and only if $M$ has a symmetric or c-symmetric spectrum. Finally, we discuss the relationship between a matrix $M$ being antidiagonalizable, $M$ being nilpotent, and the ranks of generalized eigenvectors of $M$. In Section \ref{Subsect:normalAntidiag} we discuss normal antidiagonalizable matrices. In particular, we provide a characterization of the unitary diagonalization of antidiagonal matrices, which yields a sufficient condition for when unitarily antidiagonalizable matrices are normal, and therefore unitarily duodiagonalizable. Since much discussion is provided on the diagonalization of duodiagonalizable matrices in preceding sections, in Section \ref{Subsect:symAndAntisym} we discuss antidiagonalizations and unitary antidiagonalizations of duodiagonalizable matrices. The greater freedom available for antidiagonalization compared to diagonalization allows us to provide a symmetric antidiagonalization and two antisymmetric antidiagonalizations of a general duodiagonalizable matrix. All three antidiagonalizations are unitary if and only if the original matrix is normal. Finally, in Section \ref{Subsect:centro} we show the relationship centrosymmetric matrices have to diagonalizations and antidiagonalizations that generalizes the relationship exchange matrices have with diagonal matrices and antidiagonal matrices from Section \ref{Sect:props}. In particular, we show how centrosymmetric matrices allow us to transform problems about antidiagonalizable matrices into problems about diagonalizable matrices, and vice versa.

The final section, Section \ref{Sect:unitDirectSum}, begins with the Schur decomposition\footnote{see Definition \ref{D:schurDef}} of a $2 \times 2$ antidiagonal matrix with maximal degrees of freedom in the sense that no other Schur decomposition in more variables exists where all variables are independent. We then use this to derive an explicit quasidiagonal Schur decomposition and unitary quasidiagonalization of antidiagonal matrices. This result generalizes to a quasidiagonal Schur decomposition of unitarily antidiagonalizable matrices. Finally, we derive our third and final direct sum decomposition -- the direct-sum decomposition of an antidiagonalizable operator up to unitary similarity.

The three major direct sum decompositions and their corresponding quasidiagonalizations, as well as nearly all other similarity decompositions presented, are provided with explicit expressions for their corresponding similarity transformation operators. We also show how the similarity transformation operators are uniquely associated to the decompositions, and some even have no dependency on the operator being transformed. We explore the pattern to their structure but leave a more complete discussion for future research.

The direct sum decompositions are especially useful in quantum mechanics and quantum field theory, where symmetries are profound and pervasive. In particular, such decompositions are pertinent to irreducible representations of symmetry groups, which function as conceptual units of interpretation. \cite{nH18} Hilbert spaces can be decomposed into direct sums using observables and symmetries as the starting point. Moreover, ``For models with symmetry, the properties of irreducible representations constrain the possibilities of Hilbert space arithmetic, i.e. how a Hilbert space can be decomposed into sums of subspaces and factored into products of subspaces. Partitioning the Hilbert space is equivalent to parsing the system into subsystems, and these emergent subsystems provide insight into the kinematics, dynamics, and informatics of a quantum model.'' \cite{nH18}

All matrices in this paper are square matrices and all linear transformations are linear operators (isomorphisms) unless specified otherwise. Thus, under this implicit global assumption, we will omit the descriptor ``square''. A matrix of \emph{size} $n$ refers to an $n \times n$ square matrix. Many, if not most of the statements proven in this paper are true for matrices over a general field and can be extended to nonsquare matrices in some way. However, we leave such an exploration for future research, and this paper is primarily concerned with square matrices and linear operators over the field of complex numbers. Thus, if the domain of a matrix is not specified, we will assume it is a complex matrix, though we do discuss matrices over the field of real numbers, particularly in Section \ref{Sect:permDirectSum}.

\section{Preliminaries and Definitions} \label{Sect:prelim} 

We must first conjure some definitions related to matrix similarity.

\begin{definition} [Matrix Similarity] \label{D:sim}
	Square matrices $M$ and $N$ are {\bfseries similar} if and only if $M = V N V^{-1}$ for some nonsingular square matrix $V$. We call $V$ the {\bfseries similarity transformation matrix}. $M$ and $N$ are {\bfseries unitarily similar} if and only if $V$ is a unitary matrix. $M$ and $N$ are {\bfseries orthogonally similar} if and only if $V$ is an orthogonal matrix. $M$ and $N$ are {\bfseries permutation-similar} if and only if $V$ is a permutation matrix.
\end{definition}

Matrix similarity is foundational to linear algebra, as linear transformations expressed as matrices are defined up to similarity. Geometrically, unequal matrices that are similar express the same linear transformation with respect to different bases. 

Like similarity, unitary similarity corresponds to a change of basis, but in particular, a change from one orthonormal basis to another. These transformations are especially important in quantum mechanics where wavefunctions evolve unitarily, and unitary transformations preserve norms and probability amplitudes, so it is common to work in orthonormal bases. Unitary similarity transformations are also important in numerical linear algebra as their preservation of norms implies they tend to have higher numerical stability and better accuracy than nonunitary similarity transformations. In particular, the condition number of any unitary transformation is 1. Note also two matrices $M_1$ and $M_2$ are unitarily similar if and only if $(M_1,M_1^*)$ and $(M_2,M_2^*)$ are \emph{simultaneously similar}, that is, there is a nonsingular matrix $S$ such that $S M_1 S^{-1} = M_2$ and $S M_1^* S^{-1} = M_2^*$. \cite{hS91}

Recall spectral properties, such as the determinant, spectrum, characteristic polynomial, and trace as well as elementary divisors, invariant factors, minimal polynomial, Jordan canonical form, and rational canonical form are similarity invariants. In addition to all similarity invariants, singular values and the Schur normal form as well as being normal, symmetric, antisymmetric, Hermitian, and antihermitian are unitary similarity invariants. \cite{hS91}

Permutation-similarity draws an even finer distinction, as the elements of two permutation-similar matrices are the same though their positions differ, making the values of the elements permutation-similarity invariants. Because permutation matrices are orthogonal and unitary, \sloppy{permutation-similar} matrices are orthogonally similar and unitarily similar as well. This strength implies compositions of permutation-similarity transformations with unitary/orthogonal similarity transformations are unitary/orthogonal similarity transformations. Geometrically speaking, a permutation-similarity transformation amounts to essentially a relabeling/permuting of axes in a similarity-invariant way, preserving norms and orthogonality. Permutation-similarity is especially important in graph theory, as two (directed or undirected) graphs are isomorphic if and only if their adjacency matrices are permutation-similar.

The meaning of further variants, such as \emph{centrosymmetric similarity} or \emph{special orthogonal similarity}, should be evident.

\begin{definition} [Matrix Diagonalizability] \label{D:diag}
	A matrix $M$ is {\bfseries diagonalizable} if and only if $M$ is similar to a diagonal matrix $D$. The corresponding similarity decomposition $M = V D V^{-1}$ is a {\bfseries diagonalization} of $M$. $M$ is {\bfseries unitarily diagonalizable} if and only if $M$ is unitarily similar to a diagonal matrix $D$, and the corresponding unitary similarity decomposition $M = U D U^{-1}$ is a {\bfseries unitary diagonalization} of $M$. \cite{gDA03}
\end{definition}

The spectral theorem for normal matrices implies a matrix is normal if and only if it is unitarily diagonalizable. \cite{gDA03}

\emph{Quasidiagonal} matrices are the matrices that are, in a sense, closest to being diagonal without necessarily being diagonal.

\begin{definition} [Quasidiagonal Matrix] \label{D:quasidiag}
	A square matrix is {\bfseries quasidiagonal} if and only if it is a square-block diagonal matrix whose diagonal blocks are of size at most 2. \cite{bS06}
	\end{definition}
	
Notice quasidiagonal matrices are tridiagonal, but there are tridiagonal matrices that are not quasidiagonal. In particular, the super diagonal and sub diagonal of a quasidiagonal matrix cannot have two consecutive nonzero elements.

Quasidiagonal operators and their relationship with \emph{quasitriangular}\footnote{see Definition \ref{D:upQuasiTri}} operators in Hilbert spaces of infinite dimension, as well as topological properties, are discussed in \cite{gL75}.

The meanings of the terms \emph{quasidiagonalizable}, \emph{quasidiagonalization}, \emph{unitarily quasidiagonalizable}, and \emph{unitary quasidiagonalization} mirror the definitions given in Definition \ref{D:diag}.

Recall the string of elements perpendicular to the main diagonal of a matrix is the \emph{antidiagonal} of the matrix. We call a matrix whose only nonzero elements, if any, lie along the antidiagonal an \emph{antidiagonal matrix}.

\begin{definition} [Matrix Antidiagonalizability] \label{D:antidiag}
	A matrix $M$ is {\bfseries antidiagonalizable} if and only if $M$ is similar to an antidiagonal matrix $A$. The corresponding similarity decomposition $M = V A V^{-1}$ is an {\bfseries antidiagonalization} of $M$. $M$ is {\bfseries unitarily antidiagonalizable} if and only if $M$ is unitarily similar to an antidiagonal matrix $A$, and the corresponding unitary similarity decomposition $M = U A U^{-1}$ is a {\bfseries unitary antidiagonalization} of $M$.
\end{definition}

The matrices in a set of matrices are \emph{simultaneously diagonalizable} if and only if all matrices in the set are diagonalized by the same similarity matrix and are \emph{simultaneously antidiagonalizable} if and only if all matrices in the set are antidiagonalized by the same similarity matrix; the matrices are \emph{simultaneously unitarily diagonalizable} and, respectively, \emph{simultaneously unitarily antidiagonalizable} when the similarity matrix is unitary.

As we will see, there is much to say about matrices that are both diagonalizable and antidiagonalizable. We call such matrices \emph{duodiagonalizable}.

\begin{definition} [Duodiagonalizable] \label{D:duodiag}
	A matrix is {\bfseries duodiagonalizable} if and only if it is diagonalizable and antidiagonalizable.
\end{definition}

We distinguish duodiagonalizable from \emph{bidiagonalizable}, whereby a matrix is bidiagonalizable if and only if it is similar to a matrix whose nonzero elements lie on the main diagonal and only the super diagonal or only the sub diagonal.

We call a matrix with a main diagonal consisting only of $0$s a \emph{hollow matrix}. \cite{aN18,tD20} 

\begin{definition} [Matrix Hollowizability] \label{D:hollowize}
	A matrix $M$ is {\bfseries hollowizable} if and only if $M$ is similar to a hollow matrix $H$. The corresponding similarity decomposition $M = V H V^{-1}$ is a {\bfseries hollowization} of $M$. $M$ is {\bfseries unitarily hollowizable} if and only if $M$ is unitarily similar to a hollow matrix $H$, and the corresponding unitary similarity decomposition $M = U H U^{-1}$ is a {\bfseries unitary hollowization} of $M$. \cite{aN18,tD20}
\end{definition}

The definitions for \emph{simultaneously hollowizable} and \emph{simultaneously unitarily hollowizable} follow as expected. \cite{aN18}

We often call a similarity decomposition that is both a hollowization and a quasidiagonalization a \emph{hollow quasidiagonalization} and a matrix that has a hollow quasidiagonalization a \emph{hollow-quasidiagonalizable} matrix.

Whereas an eigendecomposition gives a basis of orthogonal eigenvectors, a hollowization gives a basis of orthogonal \emph{neutral vectors} -- vectors for which the quadratic form is 0. This is due to every real traceless matrix being orthogonally hollowizable, and it is one of the reasons hollow matrices and hollowizations are useful in asymptotic eigenvalue research and stabilization. \cite{pF69,tD20}

We will also require a notion that is a bit more general.

\begin{definition} [Pseudo-hollow] \label{D:pseudoHol}
	A matrix is {\bfseries pseudo-hollow} if and only if at most one element on its main diagonal is nonzero.
\end{definition}

We distinguish pseudo-hollow from \emph{almost hollow}, whereby an almost hollow matrix is traceless, and at most two elements on its main diagonal are nonzero. \cite{tD20} Notice every hollow matrix is pseudo-hollow, and a pseudo-hollow matrix is hollow if and only if it is traceless.

The meanings of variations of these terms, such as \emph{orthogonally pseudo-hollowizable} or \emph{simultaneously permutation-quasidiagonalizable}, should now be evident, and we will make use of such variations. 

\section{Antidiagonal Matrices and their Algebraic Properties} \label{Sect:props} 

Before discussing antidiagonalizable matrices, it is expedient to discuss antidiagonal matrices. We will also refer to some of their algebraic properties throughout this paper.

\begin{theorem} [Products, Inverses, and Powers of Antidiagonal Matrices] \label{T:propsOfAntiD}
	Let $k \in \mathbb{Z}$, and for any complex matrix $M$, define the reciprocal operator $\_^{-\mathds{1}}$ such that $M^{-\mathds{1}}$ is the matrix that takes the reciprocal of every nonzero element of $M$. Let $A$ and $B$ be complex antidiagonal matrices of size $n$ such that \begin{equation} \label{E:antiProd} 
A = \begin{pmatrix}
0	&			&	a_n	\\
	&	\iddots	&		\\
a_1	&			&	0
\end{pmatrix}, \hphantom{W}
B = \begin{pmatrix}
0	&			&	b_n	\\
	&	\iddots	&		\\
b_1	&			&	0
\end{pmatrix}.
\end{equation}

\begin{enumerate}[(a)]
\item The Product of Two Antidiagonal Matrices
\begin{equation} \label{E:antiProd} 
A B = \begin{pmatrix}
0	&			&	a_n	\\
	&	\iddots	&		\\
a_1	&			&	0
\end{pmatrix}
\begin{pmatrix}
0	&			&	b_n	\\
	&	\iddots	&		\\
b_1	&			&	0
\end{pmatrix}
= \begin{pmatrix}
a_n b_1	&			&0		\\
		&	\ddots	&		\\
0		&			&	a_1 b_n
\end{pmatrix}
\end{equation}
\item The Inverse of a Nonsingular Antidiagonal Matrix
\begin{equation} \label{E:antiInv} 
A^{-1} = (A^\top)^{-\mathds{1}} = (A^{-\mathds{1}})^{\top}
\end{equation}
This equation is valid, and $A$ is nonsingular, if and only if all elements $a_1,...,a_n$ on the antidiagonal are nonzero.
\item Integer Powers of Antidiagonal Matrices
\begin{equation} \label{E:antiPow}
A^k = \begin{dcases}
\begin{pmatrix}
a_n^{k/2} a_1^{k/2}	&					&		&					&0					\\
				&a_{n-1}^{k/2} a_2^{k/2}	&		&					&					\\
				&					&\ddots	&					&					\\
				&					&		&a_2^{k/2} a_{n-1}^{k/2}	&					\\
0				&					&		&					&a_1^{k/2} a_n^{k/2}
\end{pmatrix}	& \text{even $k$}	\\	\\
\begin{pmatrix}
0							&								&		&								&a_n^\frac{k-1}{2} a_1^\frac{k+1}{2}		\\
							&								&		&a_{n-1}^\frac{k-1}{2} a_2^\frac{k+1}{2}	&								\\
							&								&\iddots	&								&								\\
							&a_2^\frac{k-1}{2} a_{n-1}^\frac{k+1}{2}	&		&								&								\\
a_1^\frac{k-1}{2} a_n^\frac{k+1}{2}	&								&		&								&0
\end{pmatrix}	& \text{odd $k$}
\end{dcases}
\end{equation}
\end{enumerate}
\end{theorem}

\begin{proof}
Each part can be confirmed using straightforward matrix arithmetic and mathematical induction.
\end{proof}

Note the product of two antidiagonal matrices is a diagonal matrix, making antidiagonal matrices convenient square roots of diagonal matrices and an alternative to diagonal square roots. As a corollary, antidiagonal matrices are closed under odd products; that is, the product of an odd number of antidiagonal matrices is an antidiagonal matrix.

As we will see, the reciprocal matrix operator $\_^{-\mathds{1}}$ will play a role in the diagonalization of antidiagonal matrices. Notice nonsingular diagonal matrices satisfy (\ref{E:antiInv}) as well. 

Because the base ring $\mathbb{C}$ is commutative, $A^k$ for even $k$ is persymmetric. Note part $(c)$ of Theorem \ref{T:propsOfAntiD} includes inverses as a subcase, as (\ref{E:antiPow}) is valid for negative integers $k$.

Antidiagonal matrices of even size are hollow, in which case antidiagonalizations are hollowizations, and matrices of odd size are pseudo-hollow, in which case antidiagonalizations are pseudo-hollowizations. Since antidiagonal matrices are hollow if and only if they are traceless, antidiagonal matrices are a class of matrices for which being hollow and being hollowizable are equivalent. As we will see, traceless antidiagonalizable matrices are a particularly convenient subclass of hollowizable matrices.

The simplest nonsingular antidiagonal matrix is the \emph{exchange matrix}.

\begin{definition} [Exchange Matrix] \label{D:excMat}
	The {\bfseries exchange matrix} $E_n$ of size $n$ is the antidiagonal matrix of size $n$ whose antidiagonal consists of 1s. \cite{rH12}
\end{definition}

Equivalently, the exchange matrix of size $n$ is the unique matrix that is both an antidiagonal matrix and a permutation matrix of size $n$. As with the identity matrix, when the dimensions are understood the subscript is omitted.

It is trivial to prove $E$ is an involutory, special orthogonal, symmetric permutation matrix, so $E = E^{-1} = \overline{E} = E^* = E^\top = E^{-\mathds{1}}$ (where $\_^{-\mathds{1}}$ is the reciprocal matrix operation from Theorem \ref{T:propsOfAntiD}). $E_n$ acts on an $n$-vector by reversing the order of the vector's elements. Notice also exchange matrices are even roots of identity matrices.

Note if $A$ is an antidiagonal matrix, then $A E$ and $E A$ are diagonal matrices, and $E A E$ is an antidiagonal matrix. Similarly, if $D$ is a diagonal matrix, $D E$ and $E D$ are antidiagonal matrices, and $E D E$ is a diagonal matrix. However, we can make stronger conclusions. Both left multiplication by $E$ and right multiplication by $E$ defines an isomorphism between the set of diagonal matrices and the set of antidiagonal matrices. Moreover, conjugation by $E$ defines an automorphism over these sets. In fact, for every antidiagonal matrix $A$, there exists a unique diagonal matrix $D$ such that $A = E D$ is a QR decomposition for $A$. Notice also, if $A$ is antidiagonal and $D$ is diagonal, $A = E D$ implies $A^\top = D E$. We will make use of these observations.

Like diagonal matrices, antidiagonal matrices are \emph{generalized permutation matrices}, so can be expressed as the product of a diagonal matrix and a permutation matrix. In particular, $A$ is an antidiagonal matrix if and only if $A = D E$ for some diagonal matrix $D$.

We consider the following general form for complex antidiagonal matrices and will explicitly refer to this definition for $A$ as (\ref{E:antidiagA}) when we use it.

\begin{equation} \label{E:antidiagA} 
A = \begin{psmallmatrix}
0				&	 			& 				&			&			&			&			&	a_{n-1}	\\
				&				& 				&			&			&			&	\iddots	&			\\
				&				&				&			&			&	a_3		&			&			\\
				&				&				&			&	a_1		&			&			&			\\
				&				&				&	a_2		&			&			&			&			\\
				&				&	a_4			&			&			&			&			&			\\
\hphantom{a_{n-1}}	&	\iddots		&				&			&			&			&			&			\\
a_n				&				&				&			&			&			&			&	0
\end{psmallmatrix} \text{: even $n$, \hphantom{W}}
A = \begin{psmallmatrix}
0		&	 			&					&				&			&			&	a_n					\\
		&				&					&				&			&	\iddots	&						\\
		&				&					&				&	a_3		&			&						\\
		&				&					&	a_1			&			&			&						\\
		&				&	a_2				&				&			&			&						\\
		&	\iddots		&					&				&			&			&	\hphantom{a_{n-1}}		\\
a_{n-1}	&				&					&				&			&			&	0
\end{psmallmatrix} \text{: odd $n$}.
\end{equation} 

In general, manipulating the antidiagonal of a matrix while preserving the underlying linear transformation is more difficult than manipulating the main diagonal. For example, permuting elements along the antidiagonal via unitary similarity transformations is not readily available. This can be seen by noting unitary similarity transformations preserve the trace, but the trace can vary across permutations along the antidiagonal for matrices of odd size. However, we will see under what permutations of the antidiagonal the underlying linear transformation is preserved, and unitarily so, in the next section.

\section{Antidiagonalizable Matrix Permutation-Similarity Direct Sum Decomposition, Quasidiagonalization, and Hollowization} \label{Sect:permDirectSum} 

\subsection{Permutation-Similarity Direct Sum Decomposition}

As we will show, useful building blocks of antidiagonal matrices, and up to similarity of various kinds, antidiagonalizable matrices, are \emph{transpose pairs}.

\begin{definition} [Transpose Pair] \label{D:transPair}
	Two indexed elements of a matrix form a {\bfseries transpose pair} if and only if one element is the reflection of the other across the main diagonal. That is, $(M)_{i,j}$ and $(M)_{i',j'}$ are a transpose pair of matrix $M$ if and only if $i' = j$ and $j' = i$.
\end{definition}

Notice a transpose pair consists of two copies of the same indexed element if and only if the element lies on the main diagonal.

Transpose pairs can be divided into two mutually exclusive and collectively exhaustive categories.

\begin{definition} [Defective Transpose Pair] \label{D:defTransPair}
	A transpose pair is {\bfseries defective} if and only if one element in the pair is 0 and the other is nonzero. A transpose pair is {\bfseries nondefective} if and only if it is not defective.
\end{definition}

With this, we can now derive the first direct sum decomposition -- the direct sum decomposition up to permutation-similarity.

Every complex antidiagonal matrix can be quasidiagonalized by a single, particularly nice constant permutation matrix that preserves hollowness/pseudo-hollowness, yielding a convenient direct sum decomposition.

\begin{theorem} [Permutation-Similarity Direct Sum Decomposition, Quasidiagonalization, and Hollowization of an Antidiagonal Matrix] \label{T:permQuasi} 
Let $A$ be a complex antidiagonal matrix of size $n$.

\begin{enumerate} [(a)] 
\item If $n$ is even, then $A$ is permutation-similar to a hollow quasidiagonal matrix.

If $n$ is odd, then $A$ is permutation-similar to a pseudo-hollow quasidiagonal matrix.

In particular, if $A$ is given by (\ref{E:antidiagA}), then $A = P Q P^{-1}$, where, for even $n$,
\[
P = \begin{pmatrix}
	&		&	&		&		&1	&0	\\
	&		&	&		&\iddots	&	&	\\
	&		&1	&0		&		&	&	\\
1	&0		&	&		&		&	&	\\
0	&1		&	&		&		&	&	\\
	&		&0	&1		&		&	&	\\
	&		&	&		&\ddots	&	&	\\
	&		&	&		&		&0	&1		
\end{pmatrix} 
\text{\hphantom{w} and \hphantom{w}}
Q = \begin{pmatrix}
0	&a_1		&	&		&		&	&		\\
a_2	&0		&	&		&		&	&		\\
	&		&0	&a_3		&		&	&		\\
	&		&a_4	&0		&		&	&		\\
	&		&	&		&\ddots	&	&		\\
	&		&	&		&		&0	&a_{n-1}	\\
	&		&	&		&		&a_n	&0		\\		
\end{pmatrix},
\]
so that $P$ is an even permutation matrix, thus is special orthogonal,
and for odd $n$,
\[
P = \begin{pmatrix}
	&		&	&		&1		&0		\\
	&		&	&\iddots	&		&		\\
	&1		&0	&		&		&		\\
1	&		&	&		&		&		\\
	&0		&1	&		&		&		\\
	&		&	&\ddots	&		&		\\
	&		&	&		&0		&1		\\		
\end{pmatrix} 
\text{\hphantom{w} and \hphantom{w}}
Q = \begin{pmatrix}
a_1	&	&		&		&		&		\\
	&0	&a_3		&		&		&		\\
	&a_2	&0		&		&		&		\\
	&	&		&\ddots	&		&		\\
	&	&		&		&0		&a_n		\\
	&	&		&		&a_{n-1}	&0		\\		
\end{pmatrix},
\]
so that $P$ is an odd permutation matrix.

\item Let $A$ be a complex antidiagonal matrix of size $n$ with center element $c$ if $n$ is odd. Let $\mathcal{\hat{T}}$ be the set of transpose pairs $\tau = \{\tau_1,\tau_2\}$ on the antidiagonal of $A$. A direct sum decomposition is given by
\begin{equation}
A \permsim \mathfrak{Q}
\end{equation}
where
\[
\mathfrak{Q} = \begin{dcases} 
       \hspace{11 pt} \bigoplus\limits_{\tau \in \mathcal{\hat{T}}} \begin{psmallmatrix} 0 & \tau_1 \\ \tau_2 & 0 \end{psmallmatrix}  					& \text{even n} \\
       \smashoperator[r]{\bigoplus\limits_{\tau \in \mathcal{\hat{T}} \setminus \{c,c\}}} \begin{psmallmatrix} 0 & \tau_1 \\ \tau_2 & 0 \end{psmallmatrix} \oplus c \, (1)		& \text{odd n}
       \end{dcases}
\]

and where $(1)$ is the identity matrix of size 1.
\end{enumerate}
\end{theorem}

\begin{proof}
$(a)$ A proof by induction is particularly enlightening as it reveals the structure of the transformations involved. Let a subscript denote the size of matrices $A$, $P$, and $Q$.

It is straightforward to prove the base case $A^{}_2 = P^{}_2 Q^{}_2 P_2^{-1}$. Assume $A^{}_n = P^{}_n Q^{}_n P_n^{-1}$.
\[
\begin{aligned}
P_{n+2}^{\vphantom{-1}} Q_{n+2}^{\vphantom{-1}} P_{n+2}^{-1}
&= \begin{pNiceArray}{ccccc}
0						& \hdots 	& 0 	&1		&0		\\
\Block{3-3}<\LARGE>{P_n} 	& 		& 	& 0 		&0		\\
						&       	&   	& \vdots  	&\vdots	\\
						&       	&   	& 0		&0		\\
0  						&  \hdots	& 0  	& 0 		&1
    \end{pNiceArray}
    \begin{pNiceArray}{ccc}
\Block{1-1}<\LARGE>{Q_n} 	  & 		&    		\\
						  & 0		&a_{n+1}	\\
  						  &a_{n+2}&0	
    \end{pNiceArray}
\begin{pNiceArray}{ccccc}
0		& \Block{3-3}<\LARGE>{P_n^{-1}}	& 		& 	& 0		\\
\vdots	& 							& 		&  	& \vdots	\\
0		&     							&   		&   	&  0		\\
1		& 0							&  \hdots 	&0	&	0	\\
0 		& 0							& \hdots  	&0	&	1
    \end{pNiceArray}\\
    & = \begin{pNiceArray}{ccccc}
0							& \hdots	& 0 	&0		&a_{n+1}	\\
\Block{3-3}<\LARGE>{P_n Q_n}	& 		& 	& 0		& 0		\\
							&       	&   	& \vdots  	&   \vdots	\\
							&       	&   	& 0		&	0 	\\
0							&  \hdots	&0  	& a_{n+2}	&0
    \end{pNiceArray}
\begin{pNiceArray}{ccccc}
0		& \Block{3-3}<\LARGE>{P_n^{-1}}	&		&	& 0		\\
\vdots	& 							&		&  	& \vdots	\\
0		&     							&   		&   	&  0		\\
1		& 0							&  \hdots 	& 0	&	0	\\
0 		&  0							& \hdots  	&0	&	1
    \end{pNiceArray}\\
        & = \begin{pNiceArray}{ccc}
0		&  										&a_{n+1}	\\
		& \Block{1-1}<\LARGE>{ P^{}_n Q^{}_n P_n^{-1}}	&		\\
a_{n+2}	&										&0
    \end{pNiceArray}\\
    & = A_{n+2}
    \end{aligned}
\]
The proof for odd $n$ is essentially the same.

Finally, note the determinant of a permutation matrix is the signature of the corresponding permutation, and it is straightforward to see the parity of the permutation corresponding to $P$ is equal to the parity of $n$.

$(b)$ Part $(b)$ is essentially an abstract algebraic restatement of part $(a)$. $Q$ uniquely determines and is uniquely determined by $\mathfrak{Q}$ up to a permutation of the diagonal blocks, where permutations of the diagonal blocks are in one-to-one correspondence with permutations of $\mathcal{\hat{T}}$ treated as an ordered multiset.
\end{proof}

For our discussion of this theorem, we will need a quick and dirty, but useful, definition.

\begin{definition} [Q-Pseudo-Hollow Quasidiagonal Matrix] \label{D:QPseudoHol}
	A {\bfseries Q-pseudo-hollow quasidiagonal} matrix is a pseudo-hollow quasidiagonal matrix whose only nonzero element on the main diagonal, if it exists, forms a $1 \times 1$ block; in other words, this element is the unique nonzero element it its row, in its column, and on the main diagonal. A matrix $M$ is {\bfseries Q-pseudo-hollow quasidiagonalizable} if and only if M is similar to a Q-pseudo-hollow quasidiagonal matrix.
\end{definition}

Equivalently, a matrix is Q-pseudo-hollow quasidiagonal if and only if it is of the form for $Q$ in Theorem \ref{T:permQuasi}, and a matrix is Q-pseudo-hollow quasidiagonalizable if and only if it is similar to such $Q$.

Recall permutation-similarity is a very strong similarity isomorphism, implying orthogonal similarity as well as unitary similarity. Not only does it preserve orthogonality and normality, but the basis vectors themselves remain invariant -- only to be rearranged. Geometrically, the positions of unlabeled axes remain invariant across a permutation-similarity transformation, so there is no way to distinguish permutation-similar linear transformations in unlabeled coordinate systems (whereby coordinates are treated as multisets instead of ordered tuples). 

The permutation-similarity transformation given explicitly in Theorem \ref{T:permQuasi} defines a one-to-one correspondence between the elements of a traceless antidiagonal matrix and of a hollow quasidiagonal matrix, as well as between their basis vectors. The transformation also defines an isomorphism between two important sets of matrices; it transforms every traceless antidiagonal matrix to a unique hollow quasidiagonal matrix, and by inversion, every hollow quasidiagonal matrix to a unique traceless antidiagonal matrix. Thus, for many practical applications, hollow quasidiagonal matrices and traceless antidiagonal matrices can be considered to be ``the same''. This is especially useful because nonsingular hollow quasidiagonal matrices are the simplest nonsingular hollow matrices. Up to permutation-similarity, we can see traceless antidiagonalizations and hollow quasidiagonalizations are equivalent as well. More generally, the permutation-similarity transformation defines a one-to-one correspondence between the set of antidiagonal matrices and the set of Q-pseudo-hollow quasidiagonal matrices.

For even $n$, the permutation-similarity transformation defines a permutation-hollowization and permutation-quasidiagonalization. For odd $n$, the permutation-similarity transformation defines a permutation-pseudo-hollowization and permutation-quasidiagonalization.

Some consideration of Theorem \ref{T:permQuasi} illuminates an interesting comparison between diagonal matrices and antidiagonal matrices; whereas diagonal matrices that differ only by a permutation of their diagonal elements are unitarily similar, antidiagonal matrices that differ only by a permutation of their antidiagonal \emph{transpose pairs} (excluding the transpose pair containing the center element for antidiagonal matrices of odd size) or a transposition of the elements within each pair are unitarily similar.\footnote{Therefore, the corresponding permutation group, for any antidiagonal matrix of size $n$, is $S_{\floor{\frac{n}{2}}} \times \mathbb{Z}_2$.} We will generalize this observation from unitary similarity to similarity in the discussion following Theorem \ref{T:antiEigendec}.

Theorem \ref{T:permQuasi}, by itself, pertains only to antidiagonal matrices and quasidiagonal matrices. An immediate corollary considerably broadens its conclusions to antidiagonalizable matrices and quasidiagonalizable matrices. Recall a matrix $M$ is hollow-quasidiagonalizable if and only if $M$ is similar to a hollow quasidiagonal matrix, and $M$ is Q-pseudo-hollow-quasidiagonalizable if and only if $M$ is similar to a Q-pseudo-hollow-quasidiagonal matrix.

\begin{corollary} [Permutation-Similarity Direct Sum Decomposition of an Antidiagonalizable Matrix, Hollow Quasidiagonalizations, and Antidiagonalizations] \label{C:quasiAnti}
\hphantom{}

\begin{sloppypar}
\begin{enumerate} [(a)] 
\item A traceless matrix $M$ is antidiagonalizable if and only if $M$ is hollow-quasidiagonalizable. A traceless matrix $M$ is permutation-/orthogonally/unitarily antidiagonalizable if and only if $M$ is permutation-/orthogonally/unitarily hollow-quasidiagonalizable.

A matrix $M$ is antidiagonalizable if and only if $M$ is Q-pseudo-hollow-quasidiagonalizable. A matrix $M$ is permutation-/orthogonally/unitarily antidiagonalizable if and only $M$ is permutation-/orthogonally/unitarily Q-pseudo-hollow-quasidiagonalizable.

\item A matrix $M$ is antidiagonalizable if and only if $M$ is similar to a direct sum decomposition into hollow, $2 \times 2$ matrices with an additional $1 \times 1$ matrix if $M$ is of odd size. A matrix $M$ is permutation-/orthogonally/unitarily antidiagonalizable if and only if $M$ is permutation-/orthogonally/unitarily similar to a direct sum decomposition into hollow, $2 \times 2$ matrices with an additional $1 \times 1$ matrix if $M$ is of odd size.
\end{enumerate}
\end{sloppypar}
\end{corollary}

\begin{proof}
$(a)$ Let $M$ be a matrix that is permutation-/orthogonally/unitarily antidiagonalizable to $A$. Since $A$ is permutation-similar to $Q$ given in Theorem \ref{T:permQuasi}, and permutation-similarity entails orthogonal similarity and unitary similarity, $M$ must be permutation-/orthogonally/unitarily similar to $Q$. In the same way, if $M$ is antidiagonalizable to $A$, then $M$ must be similar to $Q$.

Let $M$ be a matrix that is permutation-/orthogonally/unitarily similar to $Q$ of the form given in Theorem \ref{T:permQuasi}. Since $Q$ is permutation-antidiagonalizable, and permutation-similarity entails orthogonal similarity and unitary similarity, $M$ must be permutation-/orthogonally/unitarily antidiagonalizable as well. Similarly, if $M$ is similar to $Q$, then $M$ must be antidiagonalizable.

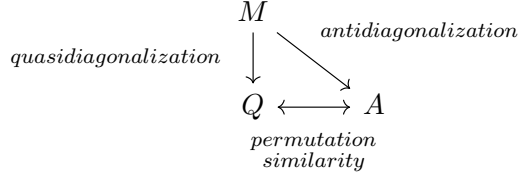
\[
\begin{tikzcd}
M \arrow[rd,"", "antidiagonalization" yshift=1ex] \arrow[d, "quasidiagonalization"' xshift = -2ex] &                  \\
Q                                \arrow[r, "permutation" yshift=-4ex,"similarity" yshift=-5.7ex,leftrightarrow] & A 
\end{tikzcd}
\]

$(b)$ This is essentially an abstract algebraic restatement of part $(a)$.
\end{proof}

Thus, antidiagonalizable matrices are quasidiagonalizable and Q-pseudo-hollow quasidiagonalizable matrices are antidiagonalizable. Additionally, due to the strength of the permutation-similarity transformation between an antidiagonal matrix $A$ and a Q-pseudo-hollow quasidiagonal matrix $Q$, the strength of a quasidiagonalization to $Q$ that factors through $A$ inherits the strength of the antidiagonalization to $A$, and the strength of an antidiagonalization to $A$ that factors through $Q$ inherits the strength of the quasidiagonalization to $Q$.

\subsection{The Real Schur Decomposition and Real Antisymmetric Matrices} \label{Subsect:realSchur} 

Theorem \ref{T:permQuasi} has important implications for real antisymmetric matrices, but it is expedient to review a few concepts first.

\begin{definition} [Schur Decomposition and Form] \label{D:schurDef}
	A {\bfseries Schur decomposition} for complex matrix $M$ is a decomposition $M = U T U^{-1}$ where $U$ is a unitary matrix and $T$, the {\bfseries Schur form} of $M$, is an upper triangular matrix whose diagonal consists of the eigenvalues of $M$. \cite{gG96}
\end{definition}

An important theorem in linear algebra proven by Issai Schur is that every complex matrix has a Schur decomposition. \cite{gG96}

\begin{definition} [Upper Quasitriangular Matrix] \label{D:upQuasiTri}
	Matrix $M$ is {\bfseries upper quasitriangular} if and only if it is a block upper triangular matrix whose blocks have size at most 2. \cite{gG96}
\end{definition}

Notice quasidiagonal matrices and upper triangular matrices are both upper quasitriangular.

A particularly important quasitriangular matrix form is the \emph{real Schur form}.

\begin{definition} [Real Schur Decomposition and Form] \label{D:realSchurDef}
	A {\bfseries real Schur decomposition} for real matrix $M$ is a decomposition $M = R T R^{-1}$ where $R$ is a real orthogonal matrix and $T$, the {\bfseries real Schur form} of $M$, is an upper quasitriangular matrix consisting of blocks of size 1 or blocks of size 2 having complex conjugate eigenvalues. \cite{gG96}
\end{definition}

Notice an unfortunate consequence of these (canonical) definitions is that a real Schur decomposition is \emph{not} necessarily a Schur decomposition. In particular, a Schur form must be upper-triangular, whereas a real Schur form need only be upper quasitriangular. Essentially, the extra restriction for a real Schur form to be real partially balances with the extra freedom granted to the real Schur form to be upper quasitriangular instead of upper triangular.

Theorem \ref{T:permQuasi} leads to the real Schur form of a real antisymmetric matrix.

\begin{corollary} [Real Schur Decomposition of an Antidiagonal Matrix] \label{C:realSchurDecomp}
Let $Q$ be the quasidiagonalization of antidiagonal matrix $A$ as in Theorem \ref{T:permQuasi}.
\begin{enumerate} [(a)]
\item $A$ is symmetric if and only if $Q$ is symmetric.
\item $A$ is antisymmetric if and only if $Q$ is antisymmetric.
\item If $A$ is real and antisymmetric, the decomposition in Theorem \ref{T:permQuasi} is the real Schur Decomposition of $A$, and $Q$ is the real Schur form of $A$.
\end{enumerate}
\end{corollary}

\begin{proof}
Parts $(a)$ and $(b)$ follow trivially.

$(c)$ Let $A$, $P$, and $Q$ be as in Theorem \ref{T:permQuasi} where $A$ is also real. Permutation matrices are real orthogonal matrices, so $P$ is a real orthogonal matrix. Since quasidiagonal matrices are upper quasidiagonal, $Q$ is upper quasidiagonal. Finally, because $Q$ is antisymmetric, it is clear the characteristic polynomials of the blocks of size 2 have roots that are complex conjugate pairs (using the conjugate root theorem).
\end{proof}

If $A$ is real but not necessarily antisymmetric, $Q$ satisfies all the conditions of being the real Schur form of $A$ except a block of size 2 may not have eigenvalues that are complex conjugate pairs. However, $Q$ will have the additional properties of being hollow/pseudo-hollow and quasidiagonal. As we will see, these properties are more useful for our purposes. 

A general Schur decomposition for a complex antidiagonal matrix is given in Theorem \ref{T:schur}.

The spectral theorem for real symmetric matrices implies every real symmetric matrix $M$ is orthogonally diagonalizable to a real matrix whose diagonal consists of the eigenvalues of $M$. There is a nice analogue to this for real antisymmetric matrices and antidiagonalization. Any real antisymmetric matrix $M$ is orthogonally antidiagonalizable to a real matrix whose antidiagonal elements are the eigenvalues of $M$ multiplied by $\imath$. In fact, the antidiagonal form can be chosen to be antisymmetric.

\begin{corollary} [Orthogonal Antidiagonalization of a Real Antisymmetric Matrix] \label{C:realAntisymAntidiag}
Every real antisymmetric matrix $M$ is orthogonally antidiagonalizable to a real antisymmetric antidiagonal matrix whose antidiagonal consists of the eigenvalues of $M$ multiplied by $\imath$.

The orthogonal antidiagonalization is special orthogonal for even $n$ and not special orthogonal for odd $n$.
\end{corollary}

\begin{proof}
Every real antisymmetric matrix $M$ of size $n$ is special orthogonally quasidiagonalizable to a real matrix of the form
\[
\begin{pmatrix}
0	&r_1		&		&		&		\\
-r_1	&0		&		&		&		\\
	&		&\ddots	&		&		\\
	&		&		&0		&r_l		\\
	&		&		&-r_l		&0
\end{pmatrix}
\text{(even n) \hphantom{w} or \hphantom{w}}
\begin{pmatrix}
0	&r_1		&		&			&			&	\\
-r_1	&0		&		&			&			&	\\
	&		&\ddots	&			&			&	\\
	&		&		&0			&r_l			&	\\
	&		&		&-r_l			&0			&	\\
	&		&		&			&			&0	\\	
\end{pmatrix}
\text{(odd n)}
\]
for $l = \frac{n}{2}$ when $n$ is even and $l = \frac{n-1}{2}$ when $n$ is odd, where $r_j \ge 0$ for all $j = 0,...,l$, and where the eigenvalues of $M$ are of the form $\pm r_j \imath$. \cite{gG96,bZ62} These matrices are clearly permutation-similar to real antisymmetric antidiagonal matrices by Theorem \ref{T:permQuasi}. The determinant of the composite orthogonal similarity transformation is the product of the determinants of the transformations being composed, so is 1 for even $n$ and -1 for odd $n$.

\end{proof}

Another proof of Corollary \ref{C:realAntisymAntidiag} is outlined at the send of Section \ref{Subsect:symAndAntisym}.

The quasidiagonalizations used in Theorem \ref{T:permQuasi}, Corollary \ref{C:realSchurDecomp}, and Corollary \ref{C:realAntisymAntidiag} are closely related to decompositions given by Youla. Let $M$ be a complex square matrix. If $M$ is symmetric, there exists a unitary matrix $U$ such that $U M U^\top$ is a diagonal matrix. \cite{dY61,jS60} This result generalizes to quaternions as well. \cite{jS60} If $M$ is antisymmetric, then there exists a unitary matrix $U$ such that $U M U^\top$ is a quasidiagonal hollow matrix. \cite{dY61}

\section{Antidiagonalizable Matrix Similarity Direct Sum Decomposition, Jordan Decomposition, and Diagonalization} \label{Sect:simDirectSum} 

\subsection{Spectral Properties of Antidiagonalizable Matrices} \label{Subsect:specProp} 

As we will see, spectral properties are useful for characterizing antidiagonalizable matrices. To this end, the following definitions are helpful.

\begin{definition} [Symmetric Spectrum] \label{D:symSpec}
	A matrix $M$ has a {\bfseries symmetric spectrum} if and only if, for every eigenvalue $\lambda$ of $M$ with algebraic multiplicity $l$, $-\lambda$ is also an eigenvalue of $M$ with algebraic multiplicity $l$. \cite{wH23,wH22}
\end{definition}

Equivalently, matrix $M$ has a symmetric spectrum if and only the spectrum of $M$ remains invariant across multiplication by $-1$. \cite{eG20,fR19}

Notice this implies $M$ is traceless. Also, if $M$ is of odd size, then $M$ is singular, as 0 must be an eigenvalue.

\begin{definition} [c-Symmetric Spectrum] \label{D:csymSpec}
	A matrix of odd size has a {\bfseries c-symmetric spectrum} $\mathcal{S}$ (or a {\bfseries center-symmetric spectrum}) if and only if $\exists c \in \mathcal{S}$ such that $\mathcal{S} \setminus \{c\}$ is symmetric. We call $c$ the {\bfseries center} of spectrum $\mathcal{S}$.
	\end{definition}

Let $M$ be a matrix of odd size. If $M$ has a c-symmetric spectrum, then $tr(M) = c$. Furthermore, $M$ having a symmetric spectrum is equivalent to $M$ having a 0-symmetric spectrum, which is equivalent to $M$ being traceless.

If a matrix $M$ with a symmetric spectrum is diagonalizable, then any diagonalization of $M$ is permutation-similar to an antipersymmetric diagonal matrix -- that is, a matrix of the form

\begin{equation} \label{E:DCentAt0}
D = \begin{psmallmatrix}
\lambda_1		& 					&					&					&			0	\\
			& 		\lambda_2		&					&	 				&	 			\\
			&					&		\ddots		&					&				\\
			&					&					&	-\lambda_2 		&				\\
0			&					&					&					&			-\lambda_1
	\end{psmallmatrix}.
\end{equation}

We present two derivations for the spectrum, determinant, and trace of an antidiagonal matrix.

\begin{theorem} [Spectral Properties of an Antidiagonal Matrix] \label{T:specAntiD} 
A complex antidiagonal matrix $A$ of general form (\ref{E:antidiagA}) has the spectrum, determinant, and trace given by
\begin{equation} \label{E: antidiagSpec}
\begin{aligned}
spec(A) &= \begin{dcases} 
      \hphantom{\{a_1} \{-\sqrt{a_1}\sqrt{a_2},\sqrt{a_1}\sqrt{a_2},...,-\sqrt{a_{n-1}}\sqrt{a_n},\sqrt{a_{n-1}}\sqrt{a_n}\}  		& \text{even $n$} \\
      \{a_1,-\sqrt{a_2}\sqrt{a_3},\sqrt{a_2}\sqrt{a_3},...,-\sqrt{a_{n-1}}\sqrt{a_n},\sqrt{a_{n-1}}\sqrt{a_n}\} 				& \text{odd $n$}
   \end{dcases}\\
   \\
\abs{A} &= (-1)^{\floor{\frac{n}{2}}} a_1 ... a_n = (-1)^{\frac{n(n-1)}{2}} a_1 ... a_n \\
\\
tr(A) &= \begin{dcases} 
      0  		& \text{even $n$} \\
      a_1		& \text{odd $n$}
   \end{dcases} = -\frac{1}{2} a_1 ((-1)^n - 1).
\end{aligned}
\end{equation}
\end{theorem}

\begin{proof}
We provide two proofs; the first uses properties of block matrices and a combinatorial argument. The second uses the results from Theorem \ref{T:permQuasi}, providing a demonstration of its utility.

\textbf{Proof 1:}
Let $n$ be odd, and define $A_1$, $A_2$, $A_3$, $A_4$ such that
\[
\lambda \mathit{I} - A = \left(\begin{array}{@{}c|c@{}}
    A_1 &A_2  \\	\hline
    A_3& A_4  \\	
  \end{array}\right)= 
  \left(\begin{array}{@{}cc|ccc@{}}
    \lambda&  		& \star 			& 		& a_n 	\\
     		& \ddots 	& \star 			&\iddots 	& 		\\	\hline
  	\star	&  \star	& \lambda-a_1 		& \star 	& \star	\\
		&\iddots  	&  \star			& \ddots 	& 		\\
a_{n-1} 	&  		&  \star			&  		& \lambda 
  \end{array}\right).
\]
Using the Schur determinant formula \cite{zF05,pP11}, and because $A_1$ is diagonal and thus, nonsingular, the characteristic polynomial of $A$ is given by the determinant
\[
\begin{aligned}
\abs{\lambda \mathit{I} - A} &= \abs{A_1} \abs{A_4 - A^{}_3 A^{-1}_1 A^{}_2}	\\
	&= \abs{\lambda \mathit{I}} \abs{A_4 - \lambda^{-1} A^{}_3 A^{}_2}		\\
	&= (\lambda - a_1) (\lambda^2 - a_2 a_3) ... (\lambda^2 - a_{n-1} a_{n}),
\end{aligned}
\]
which implies $A$ has the $a_1$-symmetric spectrum given in (\ref{E: antidiagSpec}).

The case for even $n$ can be proven with the same formula, but we will provide an attractive alternative. Let $n$ be even, and define $A_1$, $A_2$, $A_3$ such that
\[
\lambda \mathit{I} - A = \left(\begin{array}{@{}c|c@{}}
    A_1 &A_2  \\	\hline
    A_3& A_1  \\	
  \end{array}\right)= 
  \left(\begin{array}{@{}ccc|ccc@{}}
    \lambda&  		&		&			& 		& a_{n-1} 	\\
     		& \ddots 	&		&			&\iddots 	& 		\\
     		&  	 	&\lambda	&a_1			& 		& 		\\	\hline
		&  		&a_2		&\lambda		&  		& 		\\
		&\iddots  	&		&			& \ddots 	& 		\\
a_n 		&  		&		&			&  		& \lambda 
  \end{array}\right).
\]
Notice $A_1$, $A_2$, $A_3$ are square matrices of the same size, and notice $A_1 = \lambda \mathit{I}$, so $A_3$ and $A_1$ commute.\footnote{Scalar matrices commute with all matrices of the same size, since scalar matrices form the center of the algebra of matrices of the same size.} Because of this, by \cite{jS17}, the characteristic polynomial of $A$ is given by the determinant
\[
\begin{aligned}
\abs{\lambda \mathit{I} - A} &= \abs{A_1 A_1 - A_2 A_3}	\\
	&= (\lambda^2 - a_1 a_2) ... (\lambda^2 - a_{n-1} a_{n}),
\end{aligned}
\]
which implies $A$ has the symmetric spectrum given in (\ref{E: antidiagSpec}).

The determinant can be found by swapping rows to convert $A$ to a diagonal matrix, where each swap contributes a factor of -1, and then multiplying by the determinant of the remaining diagonal matrix. We can swap row $r$ with row $n - r + 1$ for $r = 1,...,\floor{\frac{n}{2}}$, giving $\floor{\frac{n}{2}}$ swaps. The determinant of the remaining diagonal matrix is simply the product $a_1,...,a_n$.

Finally, the trace is the sum of the eigenvalues, giving 0 for even $n$ and  $a_1$ for odd $n$.

\textbf{Proof 2:}
Let $Q$ be the permutation-quasidiagonalization of $A$ given in Theorem \ref{T:permQuasi}. Since $Q$ is a block diagonal matrix, its spectrum is the multiset union of the spectra of its blocks where multiplicities are additive, its determinant is the product of the determinants of its blocks, and its trace is the sum of the traces of its blocks. The $2 \times 2$ blocks of $Q$ are of the form $M = \begin{psmallmatrix} 0 & a_{j-1} \\ a_j & 0 \end{psmallmatrix}$, where $spec(M) = \{-\sqrt{a_{j-1}}\sqrt{a_j},\sqrt{a_{j-1}}\sqrt{a_j}\}$, $\abs{M} = -a_j a_{j-1}$, and $tr(M) = 0$. The $1 \times 1$ block in $Q$ of odd size has a spectrum, determinant, and trace equal to the only element present. With all this, and since $A$ and $Q$ are similar, the spectral properties of $Q$ are the spectral properties of $A$ and are given by (\ref{E: antidiagSpec}).
\end{proof}

Notice an antidiagonal matrix of even size has a symmetric spectrum, while one of odd size has a c-symmetric spectrum (the center $c = a_1$ in (\ref{E:antidiagA})). The terms ``center-symmetric'' and ``c-symmetric'' are derived from the fact that an antidiagonal matrix of odd size with a c-symmetric spectrum is structurally symmetric about the center element $c$ of the matrix. The center will always be an eigenvalue, so the matrix has a symmetric spectrum if and only if its center element is 0.

Notice also the determinant of an antidiagonal matrix is the product of the antidiagonal terms (up to a sign) akin to the way the determinant of a diagonal matrix is the product of the diagonal terms. 

We have the following natural corollary that broadens the conclusions of Theorem \ref{T:specAntiD} to antidiagonalizable matrices.

\begin{corollary} [Spectral Properties of an Antidiagonalizable Matrix] \label{C:specAntiDble} 
Let complex matrix $M$ of size $n$ be antidiagonalizable.

If $n$ is even, then $M$ has a symmetric spectrum.

If $n$ is odd, then $M$ has a c-symmetric spectrum. In this case, $M$ has a symmetric spectrum if and only if $M$ is traceless.

In particular, if $M$ is antidiagonalizable to general antidiagonal matrix $A$ given in (\ref{E:antidiagA}), then $spec(M) = spec(A)$, $\abs{M} = \abs{A}$, and $tr(M) = tr(A)$ as given by (\ref{E: antidiagSpec}).
\end{corollary}

\begin{proof}
Similar matrices share the same spectral properties.
\end{proof}

\subsection{A Characterization of the Eigendecomposition of Antidiagonal Matrices} \label{Subsect:eigendecomp} 

We now provide necessary and sufficient conditions for when a general complex antidiagonal matrix $A$ is diagonalizable as well as an explicit, convenient diagonalization for $A$ when $A$ is diagonalizable.\footnote{Complex antidiagonal matrices that are diagonalizable need not be \emph{unitarily} diagonalizable (i.e. normal). In general, relatively speaking, few are (see Theorem \ref{T:unitDiagAntiD}), offering convenient and simple examples of matrices of any size that are diagonalizable but not unitarily so.}

\begin{theorem} [Characterization of the Eigendecomposition of an Antidiagonal Matrix] \label{T:antiEigendec} 
Let $A$ be a complex antidiagonal matrix of size $n$, and let matrix function $\Lambda = \Lambda(A)$ be defined as below.
The following are equivalent.
\begin{enumerate}[(a)]
\item A is diagonalizable.
\item $\Lambda$ diagonalizes $A$ into a diagonal matrix $D$ with a symmetric spectrum for even $n$ and a c-symmetric spectrum for odd $n$.
\item No transpose pair of elements in $A$ is defective.
\item $\Lambda$ is nonsingular.
\end{enumerate}
When any of these conditions are met, whereby $D$ is a diagonalization of $A$, we have the following.
\begin{enumerate}[i.]
\item If $A$ is traceless, then $D$ is permutation-similar to an antipersymmetric diagonal matrix with the same elements, including multiplicities.
\item $A^2$ and $D^2$ are permutation-similar diagonal matrices.
\end{enumerate}
When $A$ is diagonalizable, an explicit eigendecomposition is given by $A = \Lambda D \, \Lambda^{-1}$ as follows.
For even size $n$ and general antidiagonal matrix $A$ given by (\ref{E:antidiagA}), $D$ is the diagonal spectral matrix with main diagonal
\[
(-\sqrt{a_1}\sqrt{a_2},\sqrt{a_1}\sqrt{a_2},-\sqrt{a_3}\sqrt{a_4},\sqrt{a_3}\sqrt{a_4},...,-\sqrt{a_{n-1}}\sqrt{a_n},\sqrt{a_{n-1}}\sqrt{a_n}),
\]
and the modal matrix is
\begin{equation} \label{E: evenLam}
\Lambda =    \begin{psmallmatrix}
0						&	0 					& 	0					&	0					&\hdots\vphantom{\frac{\sqrt{a_1}}{\sqrt{a_2}}}	&-\frac{\sqrt{a_{n-1}}}{\sqrt{a_n}}&\frac{\sqrt{a_{n-1}}}{\sqrt{a_n}}	\\
0						&	0					& 	0					&	0					&\iddots\vphantom{\frac{\sqrt{a_1}}{\sqrt{a_2}}}	&0						&0							\\
0						&	0					&-\frac{\sqrt{a_3}}{\sqrt{a_4}}	&\frac{\sqrt{a_3}}{\sqrt{a_4}}	&\hdots\vphantom{\frac{\sqrt{a_1}}{\sqrt{a_2}}}	&0						&0							\\
-\frac{\sqrt{a_1}}{\sqrt{a_2}}	&\frac{\sqrt{a_1}}{\sqrt{a_2}}	&	0					&	0					&\hdots\vphantom{\frac{\sqrt{a_1}}{\sqrt{a_2}}}	&\vdots					&\vdots						\\
1						&	1					&	0					&	0					&\hdots\vphantom{\frac{\sqrt{a_1}}{\sqrt{a_2}}}	&\vdots					&\vdots						\\
0						&	0					&	1					&	1					&\hdots\vphantom{\frac{\sqrt{a_1}}{\sqrt{a_2}}}	&0						&0							\\
0						&	0					&	0					&	0					&\ddots\vphantom{\frac{\sqrt{a_1}}{\sqrt{a_2}}}	&0						&0							\\
0						&	0					&	0					&	0					&\hdots\vphantom{\frac{\sqrt{a_1}}{\sqrt{a_2}}}	&1						&1	
\end{psmallmatrix},
\end{equation}
where then, for all positive odd $k$ such that $a_{k} = a_{k+1} = 0$, the pair $(- \frac{\sqrt{a_k}}{\sqrt{a_{k+1}}},1),(\frac{\sqrt{a_k}}{\sqrt{a_{k+1}}},1)$ is substituted with the choice of any pair of linearly independent vectors $\boldsymbol{w_k}, \boldsymbol{w_{k+1}}$ in $\mathbb{C}^2$.
For odd size $n$ and general antidiagonal matrix $A$ given by (\ref{E:antidiagA}), $D$ is the diagonal spectral matrix with main diagonal
\[
(a_1,-\sqrt{a_2}\sqrt{a_3},\sqrt{a_2}\sqrt{a_3},...,-\sqrt{a_{n-1}}\sqrt{a_n},\sqrt{a_{n-1}}\sqrt{a_n}),
\]
and the modal matrix is, for any choice of nonzero $\omega \in \mathbb{C}$,
\begin{equation} \label{E: oddLam}
\Lambda =    \begin{psmallmatrix}
0 				& 	0					&	0					&\hdots\vphantom{\frac{\sqrt{a_1}}{\sqrt{a_2}}}		&-\frac{\sqrt{a_n}}{\sqrt{a_{n-1}}}&\frac{\sqrt{a_n}}{\sqrt{a_{n-1}}}	\\
0				& 	0					&	0					&\iddots\vphantom{\frac{\sqrt{a_1}}{\sqrt{a_2}}}		&0						&0							\\
0				&-\frac{\sqrt{a_3}}{\sqrt{a_2}}	&\frac{\sqrt{a_3}}{\sqrt{a_2}}	&\hdots\vphantom{\frac{\sqrt{a_1}}{\sqrt{a_2}}}		&0						&0							\\
\omega			&	0					&	0					&\hdots\vphantom{\frac{\sqrt{a_1}}{\sqrt{a_2}}}		&\vdots					&\vdots						\\
0				&	1					&	1					&\hdots\vphantom{\frac{\sqrt{a_1}}{\sqrt{a_2}}}		&0						&0							\\
0				&	0					&	0					&\ddots\vphantom{\frac{\sqrt{a_1}}{\sqrt{a_2}}}		&0						&0							\\
0				&	0					&	0					&\hdots\vphantom{\frac{\sqrt{a_1}}{\sqrt{a_2}}}		&1						&1	
\end{psmallmatrix},
\end{equation}
where then, for all positive even $k$ such that $a_{k} = a_{k+1} = 0$, the pair $(- \frac{\sqrt{a_{k+1}}}{\sqrt{a_k}},1),(\frac{\sqrt{a_{k+1}}}{\sqrt{a_k}},1)$ is substituted with the choice of any pair of linearly independent vectors $\boldsymbol{w_k}, \boldsymbol{w_{k+1}}$ in $\mathbb{C}^2$.
\end{theorem}

\begin{proof} $(a) \iff (b) \iff (c)$ Notice every complex antidiagonal matrix $A$ falls under one of two mutually exclusive cases: either $\forall (A)_{i,j}, (A)_{i,j} = 0 \Rightarrow (A)_{j,i}
= 0$ (i.e. all transpose pairs along the antidiagonal are nondefective), or $\exists (A)_{i,j}, (A)_{i,j} = 0 \land (A)_{j,i} \neq 0$ (i.e. there exists a transpose pair along the antidiagonal that is defective).

Let $n$ be even, and let $A$, $\Lambda$, and $D$ be of the form given in the theorem. Let positive $k$ be odd, so the first case is equivalent to $\forall a_k = 0, a_{k+1} = 0$, and the second case is equivalent to $\exists a_k = 0$ such that $a_{k+1} \neq 0$.

We denote the eigenvalues of $D$ such that $\lambda_k = -\sqrt{a_k}\sqrt{a_{k+1}}$ and $\lambda_{k+1} = \sqrt{a_k}\sqrt{a_{k+1}}$, and we conclude $a_k = 0 \lor a_{k+1} = 0$ if and only if $\lambda_k = 0 \lor \lambda_{k+1} = 0$. This establishes a one-to-one correspondence between transpose pairs $a_k, a_{k+1}$ of elements on the antidiagonal of $A$ and pairs $\lambda_k$, $\lambda_{k+1}$ of opposite eigenvalues of $D$. Each pair of opposite eigenvalues depends only on its corresponding unique transpose pair. Looking at (\ref{E: evenLam}), it is evident the transpose pair $a_k, a_{k+1}$ is also in one-to-one correspondence with columns $k, k+1$ of $\Lambda$, which we will denote $v_k, v_{k+1}$. Each such pair of columns depends only on its corresponding unique transpose pair. In all, we can conclude each sextuple $(a_k, a_{k+1},\lambda_k, \lambda_{k+1}, v_k, v_{k+1})$ is independent of all others.

\begin{description}
   \item[Case 1: $\forall a_k = 0, a_{k+1} = 0$] for positive odd $k$
	   \begin{description}
	   \item[Case 1a:] $\forall k, a_k \neq 0$
   
	   In this subcase, none of the elements of $\Lambda$ dependent on $A$ are 0, and $D$ is nonsingular. We will use induction. Let a subscript denote the size of matrices $A$, $\Lambda$, and $D$.

It is straightforward to prove the base case $A^{}_2 = \Lambda^{}_2 D^{}_2 \Lambda_2^{-1}$. Assume $A^{}_n = \Lambda^{}_n D^{}_n \Lambda_n^{-1}$.
\[
\displayindent0pt
\displaywidth\textwidth
\begin{aligned}
\Lambda_{n+2}^{\vphantom{-1}} D_{n+2}^{\vphantom{-1}} \Lambda_{n+2}^{-1}
&= \begin{pNiceArray}{ccccc}[small]
0	& \hdots 		& 0 	&-\frac{\sqrt{a_{n+1}}}{\sqrt{a_{n+2}}}	&\frac{\sqrt{a_{n+1}}}{\sqrt{a_{n+2}}}				\\
\Block{3-3}<\LARGE>{\Lambda_n} & & & 0 & 0\vphantom{\frac{\sqrt{a_1}}{\sqrt{a_2}}}									\\
	&       		&   	& \vdots  							&   	\vdots\vphantom{\frac{\sqrt{a_1}}{\sqrt{a_2}}}	\\
	&       		&   	& 0								&	0\vphantom{\frac{\sqrt{a_1}}{\sqrt{a_2}}} 		\\
0  	&  \hdots     	& 0  	& 1 								&	1\vphantom{\frac{\sqrt{a_1}}{\sqrt{a_2}}}
    \end{pNiceArray}
    \begin{pNiceArray}{ccccc}[small]
\Block{1-1}<\LARGE>{D_n}\hphantom{A} & & & & 0   				\\
	&       		&   	&   							&   						\\
	&       		&   	& -\sqrt{a_{n+1}}\sqrt{a_{n+2}}		&	 					\\
0  	&       		&   	&  							&\sqrt{a_{n+1}}\sqrt{a_{n+2}}	
    \end{pNiceArray}
\begin{pNiceArray}{ccccc}[small]
0		& \Block{3-3}<\LARGE>{\Lambda_n^{-1}}	& & & 0					\\
\vdots	& 	& 	&  	& \vdots								\\
0		&      &   	&   	&  0\vphantom{\frac{\sqrt{a_1}}{\sqrt{a_2}}}	\\
-\frac{1}{2}\frac{\sqrt{a_{n+2}}}{\sqrt{a_{n+1}}}	& 0      						&  \hdots 	& 	0						&	\frac{1}{2}\vphantom{\frac{\sqrt{a_1}}{\sqrt{a_2}}}						\\
\frac{1}{2}\frac{\sqrt{a_{n+2}}}{\sqrt{a_{n+1}}}  	&  \hphantom{A}0\hphantom{A}     	& \hdots  	& \hphantom{A}0\hphantom{A} 		&	\hphantom{A}\frac{1}{2}\hphantom{A}\vphantom{\frac{\sqrt{a_1}}{\sqrt{a_2}}}
    \end{pNiceArray}\\
    & = \begin{pNiceArray}{ccccc}[small]
0	& \hdots 		& 0 	&a_{n+1}	&a_{n+1}	\\
\Block{3-3}<\LARGE>{\Lambda_n D_n} & & & 0 & 0\vphantom{\frac{\sqrt{a_1}}{\sqrt{a_2}}}\\
	&       		&   	& \vdots  	&   \vdots\vphantom{\frac{\sqrt{a_1}}{\sqrt{a_2}}}	\\
	&       		&   	& 0		&	0\vphantom{\frac{\sqrt{a_1}}{\sqrt{a_2}}} 	\\
\hphantom{A}0\hphantom{A}  	&  \hdots     	& \hphantom{A}0\hphantom{A}  	& -\sqrt{a_{n+1}}\sqrt{a_{n+2}} 		&\sqrt{a_{n+1}}\sqrt{a_{n+2}}
    \end{pNiceArray}
    \begin{pNiceArray}{ccccc}[small]
0		& \Block{3-3}<\LARGE>{\Lambda_n^{-1}}	& & & 0					\\
\vdots	& 	& 	&  	& \vdots								\\
0		&      &   	&   	&  0\vphantom{\frac{\sqrt{a_1}}{\sqrt{a_2}}}	\\
-\frac{1}{2}\frac{\sqrt{a_{n+2}}}{\sqrt{a_{n+1}}}	& 0      						&  \hdots 	& 	0						&	\frac{1}{2}\vphantom{\frac{\sqrt{a_1}}{\sqrt{a_2}}}						\\
\frac{1}{2}\frac{\sqrt{a_{n+2}}}{\sqrt{a_{n+1}}}  	&  \hphantom{A}0\hphantom{A}     	& \hdots  	& \hphantom{A}0\hphantom{A} 		&	\hphantom{A}\frac{1}{2}\hphantom{A}\vphantom{\frac{\sqrt{a_1}}{\sqrt{a_2}}}
    \end{pNiceArray}\\
        & = \begin{pNiceArray}{ccc}[small]
0	&  \hphantom{\sqrt{a_{n+1}}\sqrt{a_{n+2}}}\vphantom{\frac{\sqrt{a_1}}{\sqrt{a_2}}}	&a_{n+1}	\\
\vphantom{\frac{\sqrt{a_1}}{\sqrt{a_2}}}	& \Block{1-1}<\LARGE>{ \Lambda^{}_n D^{}_n \Lambda_n^{-1}} &\vphantom{\frac{\sqrt{a_1}}{\sqrt{a_2}}} \\
a_{n+2}	& \hphantom{\sqrt{a_{n+1}}\sqrt{a_{n+2}}}\vphantom{\frac{\sqrt{a_1}}{\sqrt{a_2}}}	&0
    \end{pNiceArray}\\
    & = A_{n+2}
    \end{aligned}
\]

In this subcase, since it has been proven $A$ and $D$ are similar, and $D$ is nonsingular, $A$ is nonsingular as well. \emph{This subcase is characterized by $A$ being diagonalizable and nonsingular}.\footnote{\label{lab1}As we will see in Corollary \ref{C:nonsingImpDiag}, $A$ being nonsingular implies $A$ is diagonalizable. So we can characterize this subcase simply as \emph{$A$ being nonsingular}.}
	   \item[Case 1b:] $\exists k, a_k = 0$   
	   
	   By existential instantiation, let $a_l = 0$ for some positive odd $l$. It follows by assumption that $a_{l+1} = 0$, and consequently, $\lambda_l = \lambda_{l+1} = 0$. Due to the dependence of columns $l,l+1$ in $\Lambda$ only on transpose pair $a_l,a_{l+1}$, choosing a substitution for the two nonzero components of columns $l,l+1$ of $\Lambda$ each amounts to choosing a linearly independent pair $\boldsymbol{w_l}, \boldsymbol{w_{l+1}}$ of eigenvectors for 
$\begin{psmallmatrix}
0 & a_{l+1} \\
a_l & 0 
\end{psmallmatrix}$
where $a_l = a_{l+1} = 0$, which can always be done. Notice, though the matrix is 0 and the eigenvalues are 0, the pair $\boldsymbol{w_l}, \boldsymbol{w_{l+1}}$ being linearly independent is necessary for $\Lambda$ to be nonsingular, which is necessary for the diagonalization desired. \emph{This subcase is characterized by $A$ being diagonalizable and singular}.
\end{description}

It is now evident that $v_k, v_{k+1}$ can be identified with the eigenvectors for eigenvalues $\lambda_k, \lambda_{k+1}$, and equivalently, $\Lambda$ is the modal matrix for the diagonalization of $A$.\footnote{We can think of $v_k, v_{k+1}$ as eigenvectors and could typeface them in bold to be consistent with our notational standard, but we sometimes prefer to regard them as columns.}

   \item[Case 2: $\exists a_k = 0, a_{k+1} \neq 0$] for positive odd $k$
   
The proof for this case follows a similar line of reasoning as that of Case 1b. Again, by existential instantiation, let $a_l = 0$ for some positive odd $l$ so that, by assumption, $a_{l+1} \neq 0$. It still follows that $\lambda_l = \lambda_{l+1} = 0$. Due to the dependence of columns $l,l+1$ in $\Lambda$ only on transpose pair $a_l,a_{l+1}$, the two 2-element columns consisting of the nonzero components of columns $l,l+1$ each, under some choice of or substitution to their entries, must form an eigenvector of $\begin{psmallmatrix}
0 & a_{l+1} \\
a_l & 0 
\end{psmallmatrix}$. These eigenvectors cannot be linearly independent as the nullspace of the matrix is one-dimensional. \emph{This case is characterized by $A$ being defective and singular}\footnote{equivalently, just defective; see footnote \ref{lab1}}. 
   \end{description}
   
Since the cases are mutually exclusive and collectively exhaustive, we have characterized all complex antidiagonal matrices into being diagonalizable or defective.
   
The proof when $n$ is odd is essentially the same where positive $k$ and $l$ are now even. In this decomposition, $\omega$ from $\Lambda$ and $\omega^{-1}$ from $\Lambda^{-1}$ cancel out, leaving $a_1$, making the decomposition independent of the value of $a_1$ and of $\omega$, except $\omega$ must be nonzero. So the analogous cases and subcases do not depend on center element $a_1$. The value of $a_1$ has no effect on the diagonalization of $A$ or on whether or not $A$ can be diagonalized, but recall it is an eigenvalue of $A$, so it does affect the invertibility of $A$. Thus, for odd $n$, Case 1 is characterized by \emph{$A$ being diagonalizable}, where Case 1a and Case 1b are no longer distinguished by the invertibility of $A$, and Case 2 is still characterized by \emph{$A$ being defective and singular}\footnote{equivalently, just defective; see footnote \ref{lab1}}.
   
$(c) \iff (d)$ Let $n$ be even and let positive $k$ be odd. Looking at (\ref{E: evenLam}) and the definition of $\Lambda$, it is evident that columns $k,k+1$ of $\Lambda$ are linearly independent if and only if $a_k,a_{k+1}$ are both nonzero or both zero. Moreover, for all $A$, columns $k,k+1$ are each linearly independent of all other columns. This implies the conclusion.

Again, the proof for the case where $n$ is odd is essentially the same where positive $k$ and $l$ are now even, and the proof does not depend on the value of center element $a_1$.

$i.$ Any diagonal matrix can have its main diagonal elements permuted without restriction via permutation-similarity transformations.

$ii.$ Let $n$ be even. Using Theorem \ref{T:propsOfAntiD}, $A^2$ is the diagonal matrix with main diagonal $(a_{n-1} a_n,...,a_3 a_4, a_1 a_2, a_1 a_2, a_3 a_4,...,a_{n-1} a_n)$, and $D^2$ is the diagonal matrix with main diagonal $(a_1a_2,a_1 a_2,a_3 a_4,a_3 a_4,...,a_{n-1} a_n,a_{n-1} a_n)$. These are clearly permutation-similar.

The case where $n$ is odd is again, essentially the same, taking into account the unique freedom of center element $a_1$.
\end{proof}

Let $A$ be a complex antidiagonal matrix of size $n$. We can conclude from Theorem \ref{T:antiEigendec}, if $A$ is diagonalizable, then $\Lambda$ diagonalizes it. In particular, every diagonalizable $A$ is diagonalizable to diagonal matrix $D$ with a symmetric spectrum for even $n$ and a c-symmetric spectrum for odd $n$ through modal matrix function $\Lambda = \Lambda(A)$. Moreover, if $A$ is traceless, we can choose $D$ to be antipersymmetric by composing $\Lambda$ with appropriate permutation matrices (permutation matrices are unitary, so the composition defines a unitary similarity transformation if and only if $\Lambda$ is unitary). As for defective antidiagonal matrices, Theorem \ref{T:antiEigendec} also shows the set of defective antidiagonal matrices is precisely the set of antidiagonal matrices containing a defective transpose pair on their antidiagonal.

The Jordan decomposition given in Theorem \ref{T:jorAntiD} generalizes the eigendecomposition given in Theorem \ref{T:antiEigendec}.

Due to the importance of $\Lambda$ and its cousins throughout this paper, we note a relationship between the transpose of $\Lambda$, the inverse of $\Lambda$, and the multiplicative inverse of all nonzero elements of $\Lambda$ (this defines the reciprocal matrix operation $\_^{-\mathds{1}}$ from Theorem \ref{T:propsOfAntiD}).

\begin{corollary} [Inverse of $\Lambda$] \label{C:lamInv}
Let $\Lambda = \Lambda(A)$ as given in Theorem \ref{T:antiEigendec} be nonsingular and $\_^{-\mathds{1}}$ as given in Theorem \ref{T:propsOfAntiD}. If $A$ is nonsingular, then the inverse of $\Lambda$ is given by
\begin{equation}\label{E:lambInvEq}
\Lambda^{-1} = \frac{1}{2} (\Lambda^\top)^{-\mathds{1}} = \frac{1}{2} (\Lambda^{-\mathds{1}})^\top.
\end{equation}

In this case, the diagonalization from Theorem \ref{T:antiEigendec} can be rewritten $A = \Lambda D \, \Lambda^{-1} = \frac{1}{2} \Lambda D \, (\Lambda^\top)^{-\mathds{1}}$.
\end{corollary}

\begin{proof}
This can be confirmed using straightforward matrix arithmetic and mathematical induction.
\end{proof}

Necessary and sufficient conditions for $A$ and $\Lambda$ to be nonsingular are given in Theorem \ref{T:propsOfAntiD} and Theorem \ref{T:antiEigendec}, respectively. Note $A$ being singular does not necessarily imply $\Lambda$ is singular. Thus, in cases where $A$ is singular but $\Lambda$ is nonsingular, $\Lambda^{-1}$ is defined but not necessarily given by (\ref{E:lambInvEq}).

Notice the remarkable similarity of the relationship given in Corollary \ref{C:lamInv} to that of the inverse of a nonsingular antidiagonal matrix given by (\ref{E:antiInv}). This relationship also makes computing $\Lambda^{-1}$ far less computationally expensive than general inverse-computing algorithms.

\subsection{Similarity Direct Sum Decomposition and Jordan Canonical Form} \label{Subsect:jordan} 

We must conjure a definition before proceeding.

\begin{definition} [Generalized Eigenvector] \label{D:genEigenV}
	A vector $\boldsymbol{x_r}$ is a {\bfseries generalized eigenvector of rank (or type) $r$} corresponding to matrix $M$ and eigenvalue $\lambda$ if and only if $(M - \lambda I)^r \boldsymbol{x_r} = 0$ but ${(M - \lambda I)^{r-1} \boldsymbol{x_r} \neq 0}$.\footnote{Note $\boldsymbol{x}$ is an eigenvector of matrix $M$ if and only if $\boldsymbol{x}$ is a generalized eigenvector of rank 1 of $M$.} \cite{rB70}
	\end{definition}
	
Now we can derive the Jordan decomposition and our second direct sum decomposition -- the direct sum decomposition up to similarity.

\begin{theorem} [Jordan Canonical Form and Similarity Direct Sum Decomposition of an Antidiagonal Matrix] \label{T:jorAntiD} 
Let $A$ be a complex antidiagonal matrix of size $n$ with center element $c$ if $n$ is odd. Let $\mathcal{T}$ be the set of nondefective transpose pairs in $A$, and let $\lambda_t$ be any of the eigenvalues associated with transpose pair $t$ (that is, $\lambda_t = \sqrt{\tau_1}\sqrt{\tau_2}$ or $\lambda_t = - \sqrt{\tau_1}\sqrt{\tau_2}$, for $\tau_1,\tau_2 \in t$)

\begin{enumerate} [(a)] 

\item The Jordan canonical form of $A$, up to a permutation of Jordan blocks, is

\begin{equation} \label{E: jor}
J =    \begin{pNiceArray}{cccccccccc}
0	&	1 				& 		&		&		&			&			&\Block{3-3}<\Huge>{0}	&				&			\\
	&	0				&0 		&		&		&			&			&					&				&			\\
	&					&\ddots	&\ddots	&		&			&			&					&				&			\\
	&					&		&	0	&1		&			&			&					&				&			\\
	&					&		&		&0		&0			&			&					&				&			\\	\hline
	&					&		&		&		&\lambda_{t_1}	&0			&					&				&			\\
	&					&		&		&		&			&\lambda_{t_2}	&0					&				&			\\
	&\Block{3-3}<\Huge>{0}	&		&		&		&			&			&\ddots				&\ddots			&			\\
	&					&		&		&		&			&			&					&-\lambda_{t_2}	&0			\\
	&					&		&		&		&			&			&					&				&-\lambda_{t_1}	
\end{pNiceArray},
\end{equation}

where $J$ is quasidiagonal, $t_i \in \mathcal{T}$, and for odd $n$, $\lambda_{t_{\abs{\mathcal{T}}}} = c$.

If $A$ is in the general form given by (\ref{E:antidiagA}), a generalized modal matrix giving a canonical basis is $\Lambda_G$, defined to be $\Lambda$ as in Theorem \ref{T:antiEigendec} modified such that, for all defective transpose pairs $a_o,a_e$ with nonzero element $\tau$, the ordered pair $((- \frac{\sqrt{a_o}}{\sqrt{a_e}},1),(\frac{\sqrt{a_o}}{\sqrt{a_e}},1))$ is substituted with the ordered pair $((x,0),(y,\frac{x}{\tau}))$ for any choice of $x \in \mathbb{C} \setminus \{0\}$ and $y \in \mathbb{C}$.

\item A direct sum decomposition is given by
\begin{equation}
A \simeq \mathfrak{N} \oplus \mathfrak{D}
\end{equation}
where the nilpotent part is

\[
\mathfrak{N} = \bigoplus\limits_{t \in \mathcal{T}^\complement} N_t \text{ where each } N_t = \begin{psmallmatrix} 0 & 1 \\ 0 & 0 \end{psmallmatrix},
\]

and the diagonal part is

\[
\mathfrak{D} = \begin{dcases} 
       \hspace{10 pt} \bigoplus\limits_{t \in \mathcal{T}} \lambda_t \begin{psmallmatrix} 1 & 0 \\ 0 & -1 \end{psmallmatrix}  								& \text{even n} \\
       \smashoperator[r]{\bigoplus\limits_{t \in \mathcal{T} \setminus \{c,c\}}} \lambda_t \begin{psmallmatrix} 1 & 0 \\ 0 & -1 \end{psmallmatrix} \oplus c \, (1)		& \text{odd n}
       \end{dcases}
\]

where $(1)$ is the identity matrix of size 1.
\end{enumerate}
\end{theorem}

\begin{proof}
$(a)$ Let $A$ of size $n$ be given by (\ref{E:antidiagA}), and let $\Lambda$ and $D$ be defined as in Theorem \ref{T:antiEigendec}.

First, notice $\Lambda$ is definable for all complex values of all its variables $a_k$ ($k = 1,...,n$). This is proven in Theorem \ref{T:antiEigendec} for the case where no transpose pair is defective. For the complementary case, $\Lambda$ is only undefined whenever the odd-indexed element of some transpose pair of $A$ is nonzero and the even-indexed element is 0. However, notice for any such transpose pair, the labels $a_k$ and $a_{k+1}$ of its elements can be flipped without changing the values they represent, whereby now the odd-indexed element is 0 and the even-indexed element is nonzero. The diagonalization remains invariant across this relabeling, and if this is done for every such transpose pair, $\Lambda$ is defined.

Therefore, without loss of generality, we may consider the odd-indexed element to be 0 and the even-indexed element to be nonzero in every defective transpose pair in $A$. 

The key to the proof is the one-to-one correspondence established in the proof of Theorem \ref{T:antiEigendec}. Because each sextuple $(a_k, a_{k+1},\lambda_k, \lambda_{k+1}, v_k, v_{k+1})$ is independent of all others, we know $J$ can be split into two matrix blocks (not necessarily Jordan blocks). One matrix block is a diagonal matrix block (the lower-right matrix block of $J$) with a symmetric spectrum for even $n$ and a c-symmetric spectrum for odd $n$, consisting of the nondefective eigenvalues of $A$, so its diagonal elements are indexed by $\mathcal{T}$. Its structure is determined entirely by $\Lambda$, $D$, and Case 1 in the proof of Theorem \ref{T:jorAntiD}. The remaining matrix block (the upper-left matrix block of $J$) consists of all the defective eigenvalues of $A$ along its diagonal, and we will now prove its structure is determined by the modification to $\Lambda$ defining $\Lambda_G$ and a modification to Case 2 in the proof of Theorem \ref{T:jorAntiD}.

If $A$ has no defective transpose pairs, then the theorem reverts to Theorem \ref{T:jorAntiD} and is proven. Therefore, beginning as in Case 2 in the proof of Theorem \ref{T:jorAntiD}, let $a_l,a_{l+1}$ be any defective transpose pair where $a_l = 0$ for some positive odd $l$ so $a_{l+1} \neq 0$. We see $\lambda_l = \lambda_{l+1} = 0$, as the only possible defective eigenvalue is 0. Columns $l,l+1$ in $\Lambda$ depend only on transpose pair $a_l,a_{l+1}$, and the two 2-element columns consisting of the nonzero components of columns $l,l+1$, under any choice of or substitution to their entries, cannot form two linearly independent eigenvectors of $\begin{psmallmatrix}
0 & a_{l+1} \\
a_l & 0 
\end{psmallmatrix}$, as the nullspace of this matrix is 0. However, they can form two linearly independent \emph{generalized} eigenvectors of rank 2.

The most general form of these linearly independent generalized eigenvectors is 
$\begin{psmallmatrix} x \vphantom{y} \\ 0 \vphantom{x/a}  \end{psmallmatrix}, \begin{psmallmatrix} y \\ x/a_{l+1} \end{psmallmatrix}$
for any choice of $x \in \mathbb{C} \setminus \{0\}$ and $y \in \mathbb{C}$. This is because, using Definition \ref{D:genEigenV} with $\lambda = 0$, 
\begin{equation} \label{E:genEigenvectEq}
\begin{pmatrix}
0 & a_{l+1} \\
0 & 0 
\end{pmatrix}^r 
\begin{pmatrix} x  \\ 0  \end{pmatrix} = 0
\text{\hphantom{w} and \hphantom{w}}
\begin{pmatrix}
0 & a_{l+1} \\
0 & 0 
\end{pmatrix}^r 
\begin{pmatrix} y \\ \frac{x}{a_{l+1}} \end{pmatrix} = 0,
\end{equation}
for $r=2$ but not for $r=1$.
Additionally, since
\begin{equation} \label{E:defMat}
\begin{pmatrix}
0 & a_{l+1} \\
0 & 0 
\end{pmatrix}
=
\begin{pmatrix}
x & y  \\
0 & \frac{x}{a_{l+1}} 
\end{pmatrix}
\begin{pmatrix}
0 & 1 \\
0 & 0 
\end{pmatrix}
\begin{pmatrix}
x & y  \\
0 & \frac{x}{a_{l+1}} 
\end{pmatrix}^{-1},
\end{equation}
these linearly independent generalized eigenvectors, in turn, bequeath a Jordan block $\begin{psmallmatrix}
0 & 1 \\
0 & 0 
\end{psmallmatrix}$ to $J$, and we can also see no further freedom can be added to them. This also implies the resulting modification to $\Lambda$ is $((- \frac{\sqrt{a_l}}{\sqrt{a_{l+1}}},1),(\frac{\sqrt{a_l}}{\sqrt{a_{l+1}}},1)) \mapsto ((x,0),(y,\frac{x}{a_{l+1}}))$, defining $\Lambda_G$.

Since all this is true for any defective transpose pair $a_l,a_{l+1}$, by universal generalization, it is true for all such pairs -- every defective transpose pair $a_k,a_{k+1}$ contributes exactly one Jordan block $\begin{psmallmatrix}
0 & 1 \\
0 & 0 
\end{psmallmatrix}$ to $J$. Finally, $\Lambda_G$ is a modal matrix function converting $A$ to $J$ (up to a permutation of Jordan blocks) and vice versa.

$(b)$ This is essentially an abstract algebraic restatement of part $(a)$. The top left matrix block of $J$ uniquely determines and is uniquely determined by $\mathfrak{N}$ up to a permutation of the Jordan blocks, where permutations of the Jordan blocks are in one-to-one correspondence with permutations of $\mathcal{T}$ treated as an ordered multiset. The bottom right matrix block of $J$ uniquely determines and is uniquely determined by $\mathfrak{D}$ up to a permutation in the same way.

\end{proof}

Not only is this Jordan decomposition an upper bidiagonalization (and thus, an upper triangularization), as are all Jordan decompositions, but it is another quasidiagonalization as well.

The Jordan decomposition given in Theorem \ref{T:jorAntiD} agrees with the eigendecomposition from Theorem \ref{T:antiEigendec} when no transpose pair in $A$ is defective.

Notice the similarity between Theorem \ref{T:jorAntiD}$(b)$ and the classification theorem for finitely-generated modules over principle ideal domains where the nilpotent part $\mathfrak{N}$ plays the roll of the torsion part and the diagonal part $\mathfrak{D}$ plays the roll of the free part. In fact, we can see $\mathfrak{N} \oplus \mathfrak{D}$ is a direct sum decomposition of a finitely generated $\mathbb{Z}$-module $\mathfrak{M}$ (that is, a finitely generated abelian group) under the map $(k,M) \mapsto M^k$. Letting $S$ denote a direct summand, $\mathfrak{N}$ is the torsion submodule (the ``torsion part'') since $\exists k \in \mathbb{Z}$ such that $S^k = 0$ (in particular, index $k =2$ for all such $S$), and $\mathfrak{D}$ is a free module (the ``free part'') of finite rank $\abs{\mathcal{T}}$ since $\forall k \in \mathbb{Z}, S^k \neq 0$. In the language of the classification theorem for abelian groups, $\mathfrak{N}$ consists of the primary cyclic groups and $\mathfrak{D}$ consists of the infinite cyclic groups.

We can also mine pedagogical value from Theorem \ref{T:jorAntiD}. Antidiagonal matrices (or hollow quasidiagonal matrices, via Theorem \ref{T:permQuasi}) offer simple and easy-to-generate examples of different matrices of any size that represent that same linear transformation up to unitary similarity (that is, are unitarily similar), as well as different matrices of any size that represent the same linear transformation (that is, are similar), and we have used them for this purposes in courses we have taught. Starting with any antidiagonal matrix of one's choice, a transposition of elements within any of the transpose pairs composed with any permutation between the transpose pairs (excluding the transpose pair containing the center element for antidiagonal matrices of odd size) yields another matrix that represents the same linear transformation up to unitary similarity.\footnote{See the discussion following Theorem \ref{T:permQuasi}.} Furthermore, transforming any transpose pair $(a_k, a_{k+1})$ from the original matrix to transpose pair $(b_k, b_{k+1})$ such that $\sqrt{a_k} \sqrt{a_{k+1}} = \sqrt{b_k} \sqrt{b_{k+1}}$ yields another matrix that represents the same linear transformation.\footnote{Arguably, antidiagonal matrices make for better pedagogical examples than diagonal matrices do because simple permutations of diagonal elements for diagonal matrices of any size yields matrices that are unitarily similar, and that's it. Transformations of antidiagonal matrices are richer while also remaining simple. We also, at times, find diagonal matrices misleadingly simplistic.}

For example, given antidiagonal matrix $M_1$ below, we immediately know $M_2$ is unitarily similar to $M_1$ because the transpose pairs $(2,3)$ and $(1,4)$ are permuted, and the transpose pair $(1,4)$ is transposed to $(4,1)$. Additionally, we know $M_3$ is similar to $M_1$, and therefore represents the same linear transformation, because $\sqrt{1} \sqrt{4} = \sqrt{2} \sqrt{2}$.

\[
M_1 = \left(\begin{array}{cccc}0 &  &  & 1 \\ &  & 2 &  \\ & 3 &  &  \\4 &  &  & 0\end{array}\right)
\text{\hphantom{www}}
M_2 = \left(\begin{array}{cccc}0 &  &  & 2 \\ &  & 4 &  \\ & 1 &  &  \\3 &  &  & 0\end{array}\right)
\text{\hphantom{www}}
M_3 = \left(\begin{array}{cccc}0 &  &  & 2 \\ &  & 2 &  \\ & 3 &  &  \\2 &  &  & 0\end{array}\right)
\]

Moreover, converting $M_2$ or $M_3$ to a hollow quasidiagonal form using Theorem \ref{T:permQuasi} gives examples where even the property of being antidiagonal is not preserved.

Since Theorem \ref{T:jorAntiD} pertains to antidiagonal matrices at the similarity-class level and linear operators (for us, linear isomorphisms) are defined up to their similarity-class, we can generalize the conclusions in Theorem \ref{T:jorAntiD} from antidiagonal matrices to antidiagonalizable linear operators, where antidiagonalizable linear operators over an infinite-dimensional vector space are defined by the inductive limit of the finite case.

\begin{corollary} [Jordan Canonical Form and Direct Sum Similarity Decomposition of an Antidiagonalizable Linear Operator] \label{C:jorAntiDble}
Let $M$ be a complex linear operator of size $n$. The following are equivalent.

\begin{enumerate} [(a)]

\item $M$ is antidiagonalizable.

\item $M$ can be expressed as a direct sum of traceless $2 \times 2$ matrices, with the exception of a single $1 \times 1$ matrix as an additional summand for odd $n$.

\item $M$ has a symmetric spectrum for even $n$ and a c-symmetric spectrum for odd $n$, whereby the only generalized eigenvectors of rank $\neq 1$ are of rank 2 with eigenvalues of 0.

\item $M$ can be expressed as a direct sum of a nilpotent matrix where every generalized eigenvector is of rank 2 and a diagonalizable matrix with a symmetric spectrum if $n$ is even and a c-symmetric spectrum if $n$ is odd.

\item The Jordan canonical form of $M$, up to a permutation of Jordan blocks, is given by Theorem \ref{T:jorAntiD}$(a)$.

\item $M$ has a decomposition given by Theorem \ref{T:jorAntiD}$(b)$.
\end{enumerate}\end{corollary}

\begin{proof}
It is sufficient to point out the Jordan canonical form from Theorem \ref{T:jorAntiD}$(a)$ and the direct sum decomposition from Theorem \ref{T:jorAntiD}$(b)$ remain invariant across similarity transformations, proving the equivalence of parts $(a)$, $(e)$, and $(f)$. It should be clear that parts $(c)$ and $(d)$ are essentially restatements of parts $(e)$ and $(f)$ but without the notation from Theorem \ref{T:jorAntiD}.

Finally, parts $(b)$ and $(f)$ are equivalent, as every direct summand from part $(f)$ is a traceless $2 \times 2$ matrix with the possible exception of a single $1 \times 1$ matrix as an additional summand, and every traceless $2 \times 2$ matrix is similar to $\begin{psmallmatrix} 0 & 1 \\ 0 & 0 \end{psmallmatrix}$ or $\begin{psmallmatrix} -\lambda & 0 \\ 0 & \lambda \end{psmallmatrix}$ for any $\lambda \in \mathbb{C}$, which are direct summands from part $(f)$.
\end{proof}

Corollary \ref{C:jorAntiDble}$(b)$ shows traceless $2 \times 2$ linear transformations are the building blocks for all antidiagonalizable linear transformations.\footnote{If equipped with the standard commutator for rings as a Lie bracket, traceless $2 \times 2$ complex matrices form the special linear Lie algebra $\mathfrak{sl}(2, \mathbb{C})$. Thus, up to similarity, $\mathfrak{sl}(1, \mathbb{C})$ and $\mathfrak{sl}(2, \mathbb{C})$ are the building blocks of antidiagonalizable operators. Some sort of Lie algebra structure for antidiagonalizable operators in general may exist, but it is not straightforward since a sum of antidiagonalizable operators is not necessarily antidiagonalizable. We leave this for future research. The matrix $\begin{psmallmatrix} 1 & 0 \\ 0 & -1 \end{psmallmatrix}$ in the direct summand of the diagonal part $\mathfrak{D}$ in Theorem \ref{T:jorAntiD} is the Pauli matrix $\sigma_z$ which, along with the other Pauli matrices, generate $\mathfrak{sl}(2, \mathbb{R})$ and play a foundation role in the quantum mechanics of quantum spin.} If a linear transformation can be expressed as a direct sum of traceless $2 \times 2$ matrices, along with a possible $1 \times 1$ matrix, then it is antidiagonalizable. Conversely, every antidiagonalizable linear transformation can be expressed as a direct sum of traceless $2 \times 2$ matrices, along with a possible $1 \times 1$ matrix. With this, antidiagonalizable linear transformations over infinite-dimensional vector spaces can be defined.

Recall a matrix is hollowizable if and only if it is traceless. There is an enlightening analogy that specifies the way in which traceless antidiagonalizable matrices are special cases of traceless matrices other than what can be concluded from a mere structural comparison between traceless antidiagonal matrices and hollow matrices. In a way analogous to the fact that \emph{a matrix is hollowizable if and only if it is traceless}, a traceless matrix is antidiagonalizable if and only if it is similar to a quasidiagonal matrix where each \emph{diagonal block} is traceless. In other words, up to similarity, in the same way we can identify hollowizable matrices with traceless matrices, we can identify traceless antidiagonalizable matrices with the subset of hollowizable matrices that are quasidiagonalizable into traceless blocks. The relationship is more evident for matrices that are already quasidiagonal; if $M$ is a quasidiagonal matrix, then $M$ is hollowizable if and only if $M$ is traceless, and $M$ is antidiagonalizable if and only if \emph{each diagonal block} of $M$ is traceless -- in the former case, each diagonal block is hollowizable ``all at once'', whereas in the latter case, each diagonal block is hollowizable ``separately''. With this, we can also conclude a traceless operator $M$ is antidiagonalizable if and only if $M$ can be represented as a direct sum of hollowizable matrices of size $\leq 2$.

An immediate conclusion that can be drawn from Corollary \ref{C:jorAntiDble} is that all traceless $2 \times 2$ matrices are antidiagonalizable. Moreover, we can see all nilpotent antidiagonalizable matrices have an index of nilpotency of at most 2.

\subsection{Conclusions for the Square of an Antidiagonalizable Matrix} \label{Subsect:square}

Theorem \ref{T:antiEigendec} part $(ii)$ shows the square of an antidiagonal matrix is diagonal and has some nice properties. We have similar conclusions for the more general antidiagonalizable matrices.

\begin{corollary} [Diagonalizability of the Square of an Antidiagonalizable Matrix] \label{C:antidiagSquared}
If complex matrix $M$ is antidiagonalizable, then $M^2$ is diagonalizable.
\end{corollary}

\begin{proof}
Since $M$ is antidiagonalizable, it is similar to its Jordan canonical form $J$ given in Theorem \ref{T:jorAntiD}$(a)$. It follows that $M^2$ is similar to $J^2$. However,
\begin{equation} \label{E: jorSq}
J^2 =    \begin{pNiceArray}{cccccccc}
0	& 				&		&				&				&\Block{3-3}<\Huge>{0}					&				&				\\
	&\ddots			&		&				&				&					&				&				\\
	&				&0		&				&				&					&				&				\\	
	&				&		&\lambda_{t_1}^2	&				&					&				&				\\
	&				&		&				&\lambda_{t_2}^2	&					&				&				\\
	&\Block{3-2}<\Huge>{0}				&		&				&				&\ddots				&				&				\\
	&				&		&				&				&					&\lambda_{t_2}^2	&				\\
	&				&		&				&				&					&				&\lambda_{t_1}^2	
\end{pNiceArray},
\end{equation}
which is clearly a diagonal matrix.
\end{proof}

The existence of nilpotent antidiagonalizable matrices demonstrates the converse to Corollary \ref{C:antidiagSquared} is not necessarily true. Nilpotent antidiagonalizable matrices are explored in Section \ref{Subsect:duodiag}.

It is clear (\ref{E: jorSq}) has some more nice properties to mine, but they remain invariant in general only when unitary similarity is assumed.

\begin{corollary} [Properties of $J^2$] \label{C:JSquared}
Let $M$ be a complex antidiagonalizable matrix that is unitarily similar to its Jordan canonical form.\footnote{Notice this is neither necessary nor sufficient for $M$ to be \emph{unitarily} antidiagonalizable.}
\begin{enumerate}[(a)]
\item $M^2$ is normal.
\item If each eigenvalue of $M$ is real or pure imaginary, then $M^2$ is Hermitian.
\item If $M$ has a real spectrum, then $M^2$ is a positive semidefinite Hermitian matrix.
\item If $M$ has a pure imaginary spectrum, then $M^2$ is a negative semidefinite Hermitian matrix.
\end{enumerate}
\end{corollary}

\begin{proof}
By assumption, $M$ is antidiagonalizable, and it is unitarily similar to its Jordan canonical form $J$ given in Theorem \ref{T:jorAntiD}$(a)$. It follows that $M^2$ is unitarily similar to $J^2$, which is given in (\ref{E: jorSq}). Each result follows from the preceding observations combined with the fact that unitary similarity preserves being normal and being Hermitian.
\end{proof}

\subsection{Duodiagonalizable Matrices and the Implications of Diagonalizability and Nilpotency} \label{Subsect:duodiag}

Corollary \ref{C:jorAntiDble}$(f)$ shows any antidiagonalizable linear operator is similar to a direct sum of copies of $\begin{psmallmatrix} 0 & 1 \\ 0 & 0 \end{psmallmatrix}$ and matrices of the form $\lambda_t \begin{psmallmatrix} 1 & 0 \\ 0 & -1 \end{psmallmatrix}$ (and $c \, I_1$, if n is odd) for eigenvalues $\lambda_t$. We can also see from this there are matrices that are antidiagonalizable but not diagonalizable.

\begin{corollary} [Nonsingularity Implies Diagonalizability for Antidiagonalizable Matrices] \label{C:nonsingImpDiag}
The only possible defective eigenvalue of an antidiagonalizable matrix is 0. Equivalently, if an antidiagonalizable matrix is nonsingular, then it is diagonalizable.
\end{corollary}

\begin{proof}
This is a straightforward consequence of Corollary \ref{C:jorAntiDble}$(d)$.
\end{proof}

From this, we can see the eigenvectors of rank 2 of an antidiagonalizable matrix $M$ are precisely the vectors that are in the null space (kernel) of $M^2$ but not in the null space of $M$. In fact, since $\begin{psmallmatrix}
0 & a_{l+1} \\
0 & 0 
\end{psmallmatrix}^2 = 0$ in (\ref{E:genEigenvectEq}), every linearly independent pair of vectors in $\mathbb{C}^2$ are linearly independent generalized eigenvectors of rank 2 for the matrix $\begin{psmallmatrix}
0 & a_{l+1} \\
0 & 0 
\end{psmallmatrix}$, and the only restriction on these generalized eigenvectors, besides linear independence, is given by (\ref{E:defMat}). This allows for a variation of the characterization of antidiagonalizable matrices given in Corollary \ref{C:jorAntiDble}, whereby any reference to generalized eigenvectors of rank 2 is replaced with ``vectors that are in the null space of $M^2$ but not of $M$''.

Notice, though all complex antidiagonalizable matrices have a symmetric or c-symmetric spectrum, the converse is not necessarily true. In particular, any complex matrix with a symmetric spectrum or c-symmetric spectrum where a nonzero eigenvalue has a linearly independent generalized eigenvector of rank $>1$ or a zero eigenvalue has a linearly independent generalized eigenvector of rank $>2$ will not be antidiagonalizable.

However, restricted to diagonalizable matrices, antidiagonalizability and having a symmetric or c-symmetric spectrum are equivalent. Recall, by Definition \ref{D:duodiag}, duodiagonalizable matrices are defined to be matrices that are both diagonalizable and antidiagonalizable.

\begin{corollary}[Characterization of Duodiagonalizable Matrices] \label{C:duodiagAntipersym}
Let $M$ be a complex matrix. $M$ is duodiagonalizable if and only if it is diagonalizable with a symmetric or c-symmetric spectrum.

Additionally, let $M$ be traceless. $M$ is duodiagonalizable if and only if it is similar to an antipersymmetric diagonal matrix.
\end{corollary}

\begin{proof}
$\Rightarrow$ By Corollary \ref{C:jorAntiDble}$(d)$, a matrix $M$ that is antidiagonalizable has the Jordan canonical form $J$, up to permutation of Jordan blocks, given in Theorem \ref{T:jorAntiD}$(a)$. If $M$ is also diagonalizable, then the lower right matrix block of $J$ -- the ``diagonal part'' that has a symmetric or c-symmetric spectrum -- is the entirety of $J$. Additionally, if $M$ is traceless, then $J$ is antipersymmetric.

$\Leftarrow$ If a diagonalizable matrix $M$ has a symmetric or c-symmetric spectrum, then it is similar to $D$ given in Theorem \ref{T:antiEigendec}. However, $D$ is the diagonalization of a general antidiagonal matrix. Additionally, if $M$ is traceless, then $D$ is permutation-similar to an antipersymmetric diagonal matrix with the same elements, including multiplicities, by Theorem \ref{T:antiEigendec}$(i)$.
\end{proof}

In other words, a diagonalizable matrix is antidiagonalizable if and only if it has a symmetric or c-symmetric spectrum.

More characterizations of duodiagonalizable matrices are given in Corollary \ref{C:moreDuodiagChar}.

If it wasn't clear from Theorem \ref{T:antiEigendec}, the following should now be clear.

\begin{corollary} [Similarity Direct Sum Decomposition of a Duodiagonalizable Matrix] \label{C:duodiagDecomp}
If $M$ is a complex duodiagonalizalbe matrix of size $n$, then $M$ is similar to a direct sum of traceless $2 \times 2$ diagonal matrices, with the exception of a single $1 \times 1$ matrix as an additional summand for odd $n$. In particular, $M \simeq \mathfrak{D}$ with $\mathfrak{D}$ given in Theorem \ref{T:jorAntiD}$(b)$.
\end{corollary}

\begin{proof}
This follows from Theorem \ref{T:jorAntiD} $(b)$ where $\mathfrak{N} = 0$ and Corollary \ref{C:jorAntiDble}.
\end{proof}

Since antidiagonalizable matrices are similar to a direct sum of $\mathfrak{N}$ and $\mathfrak{D}$ from Corollary \ref{C:jorAntiDble}$(f)$, and we have discussed the case where $\mathfrak{N} = 0$, there remains something to be said for the case where $\mathfrak{D} = 0$.

\begin{corollary} [Nilpotency and Antidiagonalizability] \label{C:nilAntidiag}
Let $M$ be a complex matrix. Any two of the following imply the third.
\begin{enumerate}[(a)]
\item $M$ is antidiagonalizable.
\item $M$ is nilpotent.
\item Every generalized eigenvector of $M$ is of rank 2.
\end{enumerate}
\end{corollary}

\begin{proof}
Let $M$ be antidiagonalizable, so by Corollary \ref{C:jorAntiDble}, the Jordan canonical form of $M$ is $J$ given in (\ref{E: jor}). If $M$ is nilpotent, $J$ is equal to the top left matrix block -- the nilpotent block -- for which all generalized eigenvectors are of rank 2. If, instead, we assume every generalized eigenvector of $M$ is of rank 2, then again, $J$ must be equal to the nilpotent block, so must be nilpotent.

Finally, let $M$ be a nilpotent matrix where every generalized eigenvector of $M$ is of rank 2. All such matrices have a Jordan canonical form given by the nilpotent block of $J$ from (\ref{E: jor}). But this is the Jordan canonical form of an antidiagonal matrix, so $M$ is antidiagonalizable.
\end{proof}

\subsection{Unitary Diagonalization and Normal Antidiagonal Matrices} \label{Subsect:normalAntidiag}

After discussing antidiagonalizable matrices that are diagonalizable, we are in a good position to provide a characterization of antidiagonal matrices that are unitarily diagonalizable. Recall by the spectral theorem for normal matrices, a complex matrix is unitarily diagonalizable if and only if it is normal. \cite{gDA03}

Not only is $\Lambda$ a convenient modal matrix in our diagonalization of antidiagonal matrices and essential to the definition of the generalized modal matrix function $\Lambda_G$ in converting an antidiagonal matrix to its Jordan canonical form, $\Lambda$ too plays an essential role in defining a modal matrix for the unitary diagonalization of normal antidiagonal matrices.

\begin{theorem} [Characterization of the Unitary Diagonalization of an Antidiagonal Matrix] \label{T:unitDiagAntiD}
Let $A$ be a complex antidiagonal matrix of size $n$, and let diagonal matrix $D$ be defined as in Theorem \ref{T:antiEigendec}. For odd $n$, let $C$ be the $n \times n$ identity matrix with the center element substituted with $\sqrt{2}$. Finally, let $\Lambda_U$ be defined as $\Lambda$ in Theorem \ref{T:antiEigendec} where every pair of substituting vectors $\boldsymbol{w_k}, \boldsymbol{w_{k+1}}$ from the theorem are chosen to be orthonormal under the dot product. The following are equivalent. 
\begin{enumerate} [(a)]
\item $A$ is normal.
\item Both elements of every transpose pair in $A$ have the same modulus.\footnote{Geometrically, this means each element in $A$ and its reflection across the main diagonal must be equidistant from the complex origin -- that is, differ only by a phase.}
\item For even $n$, $\frac{1}{\sqrt{2}} \Lambda_U$ is unitary. For odd $n$, $\frac{1}{\sqrt{2}} C \Lambda_U$ is unitary.
\end{enumerate}
In particular, whenever $A$ is unitarily diagonalizable, $\frac{1}{\sqrt{2}} \Lambda_U$ unitarily diagonalizes $A$ to $D$ when $n$ is even, and $\frac{1}{\sqrt{2}} C \Lambda_U$ unitarily diagonalizes $A$ to $D$ when $n$ is odd.
\end{theorem}

\begin{proof}
Let $A$ be in general form given by (\ref{E:antidiagA}).

$(a) \Rightarrow (b)$ The necessary and sufficient condition for $A$ to be normal is $A A^* = A^* A$. Expanding both sides and making an element-wise comparison shows this condition is equivalent to $a^{}_k a^*_k = a^{}_{k+1} a^*_{k+1}$ for positive odd $k$ when $n$ is even and for positive even $k$ when $n$ is odd. In both cases, the condition is equivalent to $\abs{a_k} = \abs{a_{k+1}}$.

$(b) \Rightarrow (c)$ Let $n$ be even. The necessary and sufficient condition for $\frac{1}{\sqrt{2}} \Lambda_U$ to be unitary is $(\frac{1}{\sqrt{2}} \Lambda_U)^{-1} = (\frac{1}{\sqrt{2}} \Lambda_U)^*$. Expanding and making an element-wise comparison shows this condition is equivalent to $\frac{\sqrt{a_{k+1}}}{\sqrt{a_k}} = (\frac{\sqrt{a_k}}{\sqrt{a_{k+1}}})^*$ for positive odd $k$ where $a_k, a_{k+1}$ are nonzero, which is equivalent to $\abs{a_k} = \abs{a_{k+1}}$ for the same $k$, $a_k$, and $a_{k+1}$.\footnote{Notice if some element in a transpose pair is 0, then other must be 0 for them to have the same modulus. There is nothing to address with positive odd $k$ where $a_{k} = a_{k+1} = 0$ because, by definition of $\Lambda_U$, the elements of the orthonormal vectors $\boldsymbol{w_k}, \boldsymbol{w_{k+1}}$ from Theorem \ref{T:antiEigendec} are not dependent on any $k$. That the chosen vectors be orthonormal is all that's required. The state of affairs for odd $n$ is similar.}

Similarly, let $n$ be odd. The necessary and sufficient condition for $\frac{1}{\sqrt{2}} C \Lambda_U$ to be unitary is $(\frac{1}{\sqrt{2}} C \Lambda_U)^{-1} = (\frac{1}{\sqrt{2}} C \Lambda_U)^*$. Expanding and making an element-wise comparison as before shows this condition is equivalent to $\frac{\sqrt{a_k}}{\sqrt{a_{k+1}}} = (\frac{\sqrt{a_{k+1}}}{\sqrt{a_k}})^*$ for positive even $k$ where $a_k, a_{k+1}$ are nonzero, which is equivalent to $\abs{a_k} = \abs{a_{k+1}}$ for the same $k$, $a_k$, and $a_{k+1}$.

$(c) \Rightarrow (a)$ Let diagonal matrix $D$ be defined as in Theorem \ref{T:antiEigendec}. Let $\frac{1}{\sqrt{2}} \Lambda_U$ be unitary for even $n$ and $\frac{1}{\sqrt{2}} C \Lambda_U$ be unitary for odd $n$. We will show $\frac{1}{\sqrt{2}} \Lambda_U$ diagonalizes $A$ to $D$ for even $n$ and $\frac{1}{\sqrt{2}} C \Lambda_U$ diagonalizes $A$ to $D$ for odd $n$.

Let $n$ be even. $\frac{1}{\sqrt{2}} \Lambda_U$ diagonalizes $A$ to $D$ since
\[
A = (\frac{1}{\sqrt{2}} \Lambda_U) D (\frac{1}{\sqrt{2}} \Lambda_U)^{-1} = \frac{1}{\sqrt{2}} \Lambda_U D \sqrt{2} \Lambda_U^{-1} = \Lambda_U D \Lambda_U^{-1},
\]
and $\Lambda_U$ is a special case of $\Lambda$, which already diagonalizes $A$. In particular, $\Lambda$ is defined up to a choice of a pair $\boldsymbol{w_k}, \boldsymbol{w_{k+1}}$ of linearly independent vectors in $\mathbb{C}^2$ for each positive odd $k$ where $a_{k} = a_{k+1} = 0$, and $\Lambda_U$ is defined to be the same except $\boldsymbol{w_k}, \boldsymbol{w_{k+1}}$ are chosen to be orthonormal with respect to the dot product. However, orthonormal vectors are always linearly independent.

The proof for odd $n$ is essentially the same with $\frac{1}{\sqrt{2}} C \Lambda_U$ taking the place of $\frac{1}{\sqrt{2}} \Lambda_U$.

\end{proof}

Of course, if some complex matrix $M$ is unitarily antidiagonalizable to a normal antidiagonal matrix $A$, then $M$ is normal due to the composition of unitary transformations being unitary. We can call such a matrix $M$ \emph{unitarily duodiagonalizable}. Given the similarity transformation matrix for the unitary antidiagonalization, Theorem \ref{T:unitDiagAntiD} can be used to find the modal matrix that unitarily diagonalizes $M$.

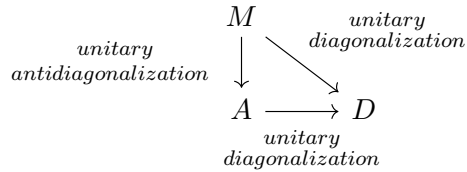
\[
\begin{tikzcd}
M \arrow[rd,"", "unitary" {xshift = 3.2 ex, yshift = 2.2 ex}, "diagonalization" yshift=0.5ex] \arrow[d, "unitary"' {xshift = -6.5 ex, yshift = 1ex}, "antidiagonalization"' {xshift = -2ex, yshift = -1 ex}] &                  \\
A                                \arrow[r, "unitary" yshift=-3.5ex,"diagonalization" yshift=-5.2ex] & D 
\end{tikzcd}
\]

\subsection{Symmetric and Antisymmetric Antidiagonalizations of Duodiagonalizable Matrices} \label{Subsect:symAndAntisym}

There is significantly more freedom in antidiagonalizing matrices that are duodiagonalizable than in diagonalizing them. Going in the reverse direction from that of Theorem \ref{T:antiEigendec} -- from diagonalization to antidiagonalization -- a duodiagonalizable matrix $M$ is antidiagonalizable in many ways. If $M$ is normal, $M$ is unitarily antidiagonalizable in many ways. In both cases, the diagonalizing matrix $\Lambda_l$ need not depend on $M$. We now provide a few of the nicest antidiagonalizations and unitary antidiagonalizations of such matrices. Recall, from Corollary \ref{C:duodiagAntipersym}, a matrix is diagonalizable and antidiagonalizable if and only if it is diagonalizable with a symmetric or c-symmetric spectrum.

\begin{corollary} [Symmetric Antidiagonalization of a Duodiagonalizable Matrix] \label{C:antidiagSym}
Let $M$ be a complex antidiagonalizable matrix of size $n$ that is diagonalizable to $D_1$ via modal matrix $V_1$. $M$ is antidiagonalizable, via a matrix $\Lambda_1$ that does not depend on $M$, to a symmetric antidiagonal matrix $A'_1$.
\begin{enumerate}[i.]
\item An explicit antidiagonalization $V_1^{-1} M V_1^{} = D_1 = \Lambda^{-1}_1 A'^{}_1 \Lambda^{}_1$ is given by setting $a_k = a_{k+1}$ in Theorem \ref{T:antiEigendec}, where positive $k$ is odd for even $n$ and even for odd $n$.
\item The antidiagonalization is unitary\footnote{in fact, special orthogonal} if and only if $M$ is unitarily diagonalizable (i.e. normal).
\end{enumerate}
\end{corollary}

\begin{proof}
If a complex matrix $M$ of size $n$ is duodiagonalizable, then $M$ diagonalizes to a diagonal matrix $D_1$ with a symmetric or c-symmetric spectrum given in Theorem \ref{T:antiEigendec} with $a_k = a_{k+1}$, where positive $k$ is odd for even $n$ and even for odd $n$. If $M$ is normal, then $M$ unitarily diagonalizes to $D_1$. The antidiagonal matrix $A$ induced from Theorem \ref{T:antiEigendec} under this specification for $a_k$ is normal, and we denote it $A'_1$. Notice the resulting $\Lambda$ does not depend on any element of $M$, $D_1$, or $A'_1$.

With this, Theorem \ref{T:unitDiagAntiD} specifies an explicit unitary diagonalization of $A'_1$ into $D_1$. Reversing this gives an explicit unitary antidiagonalization of $D_1$ into $A'_1$. Composing the unitary antidiagonalization of $D_1$ with the diagonalization of $M$ defines $\Lambda_1$ and gives an explicit antidiagonalization of $M$ that is unitary if and only if $M$ is normal.
\[
\begin{tikzcd}
M \arrow[rd,"", "antidiagonalization" yshift=1ex] \arrow[d, "diagonalization"' xshift = -2ex] &                  \\
D_1                                \arrow[r, "unitary" yshift=-4ex,"antidiagonalization" yshift=-5.7ex] & A'_1 
\end{tikzcd}
\]
\end{proof}

The assumption that $M$ is diagonalizable is necessary. If $M$ is not diagonalizable, then it has no symmetrization much less a symmetric antidiagonalization. Thus, we can conclude a matrix is duodiagonalizable if and only if it is symmetrically antidiagonalizable.

The transforming matrix $\Lambda$ from Theorem \ref{T:antiEigendec} under the assignment $a_k = a_{k+1}$, where positive $k$ is odd for even $n$ and even for odd $n$ (as specified in Corollary \ref{C:antidiagSym}) is particularly useful and enlightening, so is given below. Incidentally, there is some structural resemblance between $\Lambda$ under this assignment and adjacency matrices for binary trees.

\begin{equation} \label{lambda1}
\begin{psmallmatrix}
0		&	0 	& 	0	&	0	&\hdots\vphantom{\vdots}	&-1						&1							\\
0		&	0	& 	0	&	0	&\iddots\vphantom{\vdots}	&0						&0							\\
0		&	0	&	-1	&	1	&\hdots\vphantom{\vdots}	&0						&0							\\
-1		&	1	&	0	&	0	&\hdots\vphantom{\vdots}	&\vdots					&\vdots						\\
1		&	1	&	0	&	0	&\hdots\vphantom{\vdots}	&\vdots					&\vdots						\\
0		&	0	&	1	&	1	&\hdots\vphantom{\vdots}	&0						&0							\\
0		&	0	&	0	&	0	&\ddots\vphantom{\vdots}	&0						&0							\\
\hphantom{0}0\hphantom{0}	&\hphantom{0}0\hphantom{0}	&\hphantom{0}0\hphantom{0}	&\hphantom{0}0\hphantom{0}	&\hdots\vphantom{\vdots}	&\hphantom{0}1\hphantom{0}	&\hphantom{0}1\hphantom{0}
\end{psmallmatrix} \text{ even n \hphantom{www}}
\begin{psmallmatrix}
0 				& 	0					&	0					&\hdots\vphantom{\vdots}			&-1						&1							\\
0				& 	0					&	0					&\iddots\vphantom{\vdots}		&0						&0							\\
0				&	-1					&	1					&\hdots\vphantom{\vdots}		&0						&0							\\
1				&	0					&	0					&\hdots\vphantom{\vdots}		&\vdots					&\vdots						\\
0				&	1					&	1					&\hdots\vphantom{\vdots}		&0						&0							\\
0				&	0					&	0					&\ddots\vphantom{\vdots}			&0						&0							\\
\hphantom{0}0\hphantom{0}	&\hphantom{0}0\hphantom{0}	&\hphantom{0}0\hphantom{0}	&\hdots\vphantom{\vdots}	&\hphantom{0}1\hphantom{0}	&\hphantom{0}1\hphantom{0}	
\end{psmallmatrix} \text{ odd n}
\end{equation}

Moreover, every duodiagonalizable matrix $M$ is antidiagonalizable, via matrices $\Lambda_2,\Lambda_3$ that do not depend on $M$, to antidiagonal matrices $A'_2,A'_3$ that are antisymmetric if and only if $M$ is traceless.

\begin{corollary} [Antisymmetric Antidiagonalization of a Duodiagonalizable Matrix] \label{C:antidiagAntisym}
Let $M$ be a complex antidiagonalizable matrix of size $n$ that is diagonalizable to $D_2,D_3$ via modal matrices $V_2,V_3$, respectively. $M$ is antidiagonalizable, via matrices $\Lambda_2,\Lambda_3$ that do not depend on $M$, to antidiagonal matrices $A'_2,A'_3$ that are antisymmetric if and only if $M$ is traceless.
\begin{enumerate}[i.]
\item In Theorem \ref{T:antiEigendec}, letting positive $k$ be odd for even $n$ and even for odd $n$, 

\hspace{\parindent} setting $a_{k+1} = \imath \sqrt{a_k}$ prescribes explicit antidiagonalization $V_2^{-1} M V_2^{} = D_2 = \Lambda^{}_2 A^{}_2 \, \Lambda^{-1}_2$, and 

\hspace{\parindent} setting $a_{k+1} = - \imath \sqrt{a_k}$ prescribes explicit antidiagonalization $V_3^{-1} M V_3^{} = D_3 = \Lambda^{}_3 A^{}_3 \, \Lambda^{-1}_3$.
\item The antidiagonalizations are unitary if and only if $M$ is unitarily diagonalizable (i.e. normal).
\end{enumerate}
\end{corollary}

\begin{proof}
The proof is essentially the same as that for Corollary \ref{C:antidiagSym} except with the conditions $a_{k+1} = \imath \sqrt{a_k}$ and $a_{k+1} = - \imath \sqrt{a_k}$.
\end{proof}

Here, too, the assumption that $M$ is diagonalizable is necessary. If $M$ is not diagonalizable, then it has no antisymmetrization, so no antisymmetric antidiagonalization. Therefore, a traceless matrix is duodiagonalizable if and only if it is antisymmetrically antidiagonalizable.

We formalize the characterizations given in the preceding discussions for duodiagonalizable matrices in a corollary.

\begin{corollary} [Further Characterizations of Duodiagonalizable Matrices] \label{C:moreDuodiagChar}
A complex matrix is duodiagonalizable if and only if it is symmetrically antidiagonalizable.

A traceless complex matrix is is duodiagonalizable if and only if it is antisymmetrically antidiagonalizable.
\end{corollary}

\begin{proof}
The proof is given in Corollaries \ref{C:antidiagSym} and \ref{C:antidiagAntisym} as well as in the discussions that follow.
\end{proof}

Corollary \ref{C:antidiagAntisym} fosters another proof for Corollary \ref{C:realAntisymAntidiag}, where we conclude every real antisymmetric matrix $M$ is orthogonally antidiagonalizable to a real antisymmetric antidiagonal matrix. We outline a proof here. Real antisymmetric matrices are normal, so are orthogonally diagonalizable, and because every matrix is similar to its transpose, the structure of a real antisymmetric matrix is such that it has a symmetric spectrum or c-symmetric spectrum. Hence, it is antidiagonalizable by Corollary \ref{C:duodiagAntipersym} and is then orthogonally antidiagonalizable to a real antisymmetric antidiagonal matrix by Corollary \ref{C:antidiagAntisym}.

\subsection{Centrosymmetric Diagonalization and Antidiagonalization} \label{Subsect:centro}

Recall a matrix is \emph{centrosymmetric} if and only if it is symmetric about its center. Letting $A$ and $\Lambda$ be defined as in Theorem \ref{T:antiEigendec}, if $A$ is nonsingular, the matrix obtained by replacing all nonzero elements of $\Lambda$ with 1 (equivalently, taking the absolute value of every element in the matrices in (\ref{lambda1})) can be transformed to a centrosymmetric matrix via a permutation of its columns (which is allowed because a modal matrix remains model across any permutation of its columns). Similar conclusions can be drawn for all other derived $\Lambda$'s in this paper. Centrosymmetric matrices lurk in the background of theorems about antidiagonal matrices and antidiagonalization. This is partially investigated in \cite{aS16}. We can see this even more profoundly exemplified in Theorem \ref{T:centroUnit}.

Recall from Definition \ref{D:excMat} the exchange matrix $E$ is the antidiagonal matrix consisting of 1s along its antidiagonal. We presented $E$ as the simplest nonsingular antidiagonal matrix, but there is an important way in which $E$ is also the simplest nonsingular centrosymmetric matrix. In a way akin to how $E$ transforms problems about antidiagonal matrices into problems about diagonal matrices and vice versa (see the discussion following Definition \ref{D:excMat}), centrosymmetric matrices more generally allow us to transform problems about antidiagonalizable matrices into problems about diagonalizable matrices, and vice versa. This conversion even preserves unitarity.

\begin{theorem} [Centrosymmetric Diagonalization and Antidiagonalization] \label{T:centroUnit} 
   Let $M$ be a complex matrix, and let $C$ be a centrosymmetric matrix.
   \begin{enumerate}[(a)]
      \item $C$ diagonalizes $E M$ and $M E$ if and only if $C$ antidiagonalizes $M$.
      
      The diagonalization and antidiagonalization are both unitary if and only if $C$ is unitary.
      \item $C$ diagonalizes $M$ if and only if $C$ antidiagonalizes $E M$ and $M E$.
      
      The diagonalization and antidiagonalization are both unitary if and only if $C$ is unitary.
   \end{enumerate}
\end{theorem}

\begin{proof}
Note any matrix $C$ is centrosymmetric if and only if it commutes with $E$. \cite{mY12} With this, let $C$ be a centrosymmetric matrix.

$(a)$ $\Rightarrow$ If $C$ diagonalizes $E M$ and $M E$, then $E M = C D C^{-1}$ for some diagonal matrix $D$. However,
\[
\begin{aligned}
E M &= C D C^{-1} \\
&= C E A C^{-1} \\
&= E C A C^{-1} \\
M &= C A C^{-1},
\end{aligned}
\]
where we have used the fact that $D = E A$ for some antidiagonal matrix $A$, and $E$ is an involution. Notice if $C$ is unitary, then the antidiagonalization is unitary.

$\Leftarrow$ If $C$ antidiagonalizes $M$, then $M = C A C^{-1}$ for some antidiagonal matrix $A$. However,
\[
\begin{aligned}
M &= C A C^{-1} \\
&= C E D C^{-1} \\
&= E C D C^{-1} \\
E M &= C D C^{-1},
\end{aligned}
\]
where we have used the fact that $A = E D$ for some diagonal matrix $D$, and $E$ is an involution. Furthermore, if $C$ is unitary, then the diagonalization is unitary. Similarly,
\[
\begin{aligned}
M &= C A C^{-1} \\
&= C D E C^{-1} \\
&= C D C^{-1} E \\
M E &= C D C^{-1},
\end{aligned}
\]
where we have used the fact that $A = E D$ for some diagonal matrix $D$, $E$ is an involution, and $E$ commutes with $C^{-1}$ because the pseudoinverse (and Drazin inverse) of a centrosymmetric matrix is centrosymmetric \cite{wYP73}, so the inverse of a centrosymmetric matrix is centrosymmetric. Again, if $C$ is unitary, then the diagonalization is unitary.

$(b)$ The proof for part $(b)$ is the same as the proof for part $(a)$ under the substitution $D \leftrightarrow A$ and by swapping ``diagonalizes'' with ``antidiagonalizes''.
\end{proof}

\section{Antidiagonalizable Matrix Unitary Similarity Direct Sum Decomposition, Schur Decomposition, and Quasidiagonalization} \label{Sect:unitDirectSum} 

We can now build a general Schur decomposition \footnote{see Definition \ref{D:schurDef}}, which will also function as a unitary quasidiagonalization and a unitary upper bidiagonalization, and our third and last direct sum decomposition -- the direct sum decomposition up to unitary similarity -- from $2 \times 2$ building blocks. In the interest of theory and generalization, we find the building blocks with the most degrees of freedom possible.

\begin{theorem} [Schur Decomposition of a $2 \times 2$ Antidiagonal Matrix with Maximal Degrees of Freedom] \label{C:2by2Schur}
Consider general complex $2 \times 2$ antidiagonal matrix $A_2 = \begin{psmallmatrix}	0	&	a_1	\\	a_2	&	0	\end{psmallmatrix}$, where $a_1 = r_1 {\rm e}^{\imath \theta_1}$ and $a_2 = r_2 {\rm e}^{\imath \theta_2}$ for $r_1,r_2 \ge 0$ but not both 0, and $\theta_1,\theta_2 \in \mathbb{R}$. Without loss of generality, a Schur decomposition of $A_2$ is $T = \Gamma A_2 \Gamma^{-1}$, with unitary similarity transformation matrix
\begin{equation} \label{E: ups1}
\Gamma = \begin{pmatrix}
-{\rm e}^{\imath \phi} \sqrt{\frac{r_1}{r_1 + r_2}}								&	{\rm e}^{\imath (\phi + \frac{1}{2}(\theta_1 - \theta_2))} \sqrt{\frac{r_2}{r_1 + r_2}}	\\
{\rm e}^{\imath ((t - \phi) + \frac{1}{2}(\theta_2 - \theta_1))} \sqrt{\frac{r_2}{r_1 + r_2}}	&	{\rm e}^{\imath (t - \phi)} \sqrt{\frac{r_1}{r_1 + r_2}}
\end{pmatrix},
\end{equation}
and Schur form
\begin{equation} \label{E: s1}
T = \begin{pmatrix}
-\sqrt{a_1 a_2}	&	\hphantom{w} {\rm e}^{\imath (2 \phi + \theta_1 - t)} (r_2 - r_1)	\\
0			&	\hphantom{w} \sqrt{a_1 a_2}
\end{pmatrix},
\end{equation}
for all $\phi,t \in \mathbb{R}$.\footnote{Notice $\phi$ and $t$ are extra degrees of freedom as they do not depend on $A_2$.}

\begin{enumerate}[i.]
\item This Schur decomposition has the maximum number of degrees of freedom in the sense that no other Schur decomposition in more variables exists where all variables are independent.
\item $T =  \Gamma A_2 \Gamma^{-1} = \Gamma^{-1} A_2 \Gamma$ if and only if $\phi = \frac{1}{2} t$.
\item $\Gamma$ is an involution, making $A_2$ and $T$ unitarily involutorily similar, if and only if $\phi = t = 0$.
\end{enumerate}
\end{theorem}

\begin{proof}
Without loss of generality, a complex $2 \times 2$ unitary matrix has the general form
\begin{equation} \label{E: genUnitary}
\Gamma = \begin{pmatrix}
-w				&	z	\\
z^*{\rm e}^{\imath t}	&	w^*{\rm e}^{\imath t}
\end{pmatrix},
\end{equation}
where $w,z \in \mathbb{C}$, $t \in \mathbb{R}$, and $\abs{w} + \abs{z} = 1$. \cite{mA13} Let $w = r_w {\rm e}^{\imath \phi}$ and $z = r_z {\rm e}^{\imath \psi}$ for $r_w, r_z \ge 0$ and $\phi, \psi \in \mathbb{R}$, so that $r_w = \sqrt{1 - r_z^2}$. In order for $T$ to be upper triangular, which is necessary for it to be a Schur form for $A_2$, its bottom-left element must be 0. It is straightforward to show this condition implies 
\[
\begin{aligned}
r_z &= \sqrt{\frac{a_2 {\rm e}^{2 \imath \psi}}{a_1 {\rm e}^{2 \imath \phi} + a_2 {\rm e}^{2 \imath \psi}}}	\\
&= \sqrt{\frac{1}{1+\frac{r_1}{r_2}{\rm e}^{\imath (2(\phi-\psi) + (\theta_1-\theta_2))}}}.
\end{aligned}
\]		
With this expression for $r_z$, it is also straightforward to show $r_w,r_z \in [0,1]$ implies, without loss of generality, solving for $\psi$,
\[
\psi = \phi + \frac{1}{2}(\theta_1 - \theta_2).
\]
Substituting this into the expression for $r_z$ gives
\[
\begin{aligned}
r_z = \sqrt{\frac{r_2}{r_1+r_2}} \text{\hphantom{w} and \hphantom{w}} r_w = \sqrt{\frac{r_1}{r_1+r_2}}	.
\end{aligned}
\]
Substituting these expressions for $\psi$, $r_z$, and $r_w$ into the general expressions for $w$ and $z$, and then substituting $w$ and $z$ into (\ref{E: genUnitary}), gives (\ref{E: ups1}). With this, is straightforward to show $T = \Gamma A_2 \Gamma^{-1}$.

$i.$ Since this decomposition is derived by subjecting the general form of a complex $2 \times 2$ unitary matrix ($\ref{E: genUnitary}$) to only necessary conditions, this decomposition has the maximum number of degrees of freedom possible.

$ii.$ and $iii.$ Both of these statements follow straightforwardly by comparing $\Gamma$ to $\Gamma^{-1}$ directly.
\end{proof}

If $a_1 = a_2 = 0$, then $A_2$ is the zero matrix and its Schur decomposition is trivial. Notice $\phi = \frac{1}{2} t$ from part $(ii)$ does not imply $\Gamma$ is an involution, though it does imply $\Gamma^2$ commutes with both $A_2$ and $T$. From this theorem, we can see if some $2 \times 2$ square matrix is unitarily similar to $T$, then it is unitarily antidiagonalizable.

We will mention, in passing, some other nice particular solutions are given for $\{t = 0, \phi = \frac{1}{2} \theta_2\}$, $\{t = 0, \phi = -\frac{1}{2} \theta_1\}$, and $\{t = 0, \phi = \frac{1}{4} (\theta_2 - \theta_1)\}$.

An antidiagonal matrix is not always unitarily diagonalizable, but it is always unitarily quasidiagonalizable, as shown in Theorem \ref{T:permQuasi}. Even more, a unitary quasidiagonalization exists where each diagonal block is upper-triangular, making the unitary quasidiagonalization a Schur decomposition and a unitary upper bidiagonalization.

\begin{theorem} [Quasidiagonal Schur Decomposition and Unitary Quasidiagonalization of an Antidiagonal Matrix] \label{T:schur} 
Let $A$ be a complex antidiagonal matrix of size $n$ in general form given by (\ref{E:antidiagA}). A Schur decomposition for $A$ is given by $A = \Upsilon S \, \Upsilon^{-1}$, where $S$ is a quasidiagonal\footnote{see Definition \ref{D:quasidiag}} Schur form and $\Upsilon$ is the unitary similarity transformation matrix, as given below. Thus, $A$ is unitarily similar to $S$ via unitary similarity transformation matrix $\Upsilon$.

Let $a_k = r_k {\rm e}^{\imath \theta_k}$ be the polar form for $a_k \in \mathbb{C}$, let phase $t_k \in \mathbb{R}$, and let
\begin{gather} 
\upsilon_{k,l} = {\rm e}^{\frac{1}{2} \imath t_k} \sqrt{\frac{r_l}{r_k+r_{k+1}}}	\label{E: upsij}
\\
\Omega_{e,k} = \begin{pmatrix}	\label{E: evenOmg}
-\upsilon_{k,k}		&\upsilon_{k,k+1} \, {\rm e}^{\frac{1}{2} \imath (\theta_{k} - \theta_{k+1})}	\\
\upsilon_{k,k+1} \, {\rm e}^{\frac{1}{2} \imath (\theta_{k+1} - \theta_{k})}		&\upsilon_{k,k}	
\end{pmatrix}
\\
\Omega_{o,k} = \begin{pmatrix}	\label{E: oddOmg}
-\upsilon_{k,{k+1}}		&\upsilon_{k,k} \, {\rm e}^{\frac{1}{2} \imath (\theta_{k} - \theta_{k+1})}	\\
\upsilon_{k,k} \, {\rm e}^{\frac{1}{2} \imath (\theta_{k+1} - \theta_{k})}		&\upsilon_{k,k+1}	
\end{pmatrix},
\end{gather}
where $\Omega_{e,k}$ is $\Gamma$ from Theorem \ref{C:2by2Schur} under substitution $t \mapsto t_k$, with $\phi = \frac{1}{2} t_{k}$, and substitution for subscripts given by $1 \mapsto k$ and $2 \mapsto k+1$. An explicit Schur decomposition for $A$ is given by $A = \Upsilon S \, \Upsilon^{-1}$ as follows. For even size $n$,
\begin{equation} \label{E: evenUpsilon}
\Upsilon =    \begin{pmatrix}
0						&	0 										&\vphantom{\sqrt{\frac{a_1}{a_1+a_2}}}		&-\upsilon_{n-1,n-1}&\upsilon_{n-1,n} \, {\rm e}^{\frac{1}{2} \imath (\theta_{n-1} - \theta_n)}		\\
0						&	0										&\iddots\vphantom{\sqrt{\frac{a_1}{a_1+a_2}}}	&0								&0										\\
-\upsilon_{1,1}		&\upsilon_{1,2} \, {\rm e}^{\frac{1}{2} \imath (\theta_1 - \theta_2)}	&\vphantom{\sqrt{\frac{a_1}{a_1+a_2}}}		&\vdots							&\vdots									\\
\upsilon_{1,2} \, {\rm e}^{\frac{1}{2} \imath (\theta_2 - \theta_1)}		&\upsilon_{1,1}		&\vphantom{\sqrt{\frac{a_1}{a_1+a_2}}}		&\vdots							&\vdots									\\
0						&	0										&\ddots\vphantom{\sqrt{\frac{a_1}{a_1+a_2}}}	&0								&0										\\
0						&	0										&\vphantom{\sqrt{\frac{a_1}{a_1+a_2}}}		&\upsilon_{n-1,n} \, {\rm e}^{\frac{1}{2} \imath (\theta_n - \theta_{n-1})}	&\upsilon_{n-1,n-1}	
\end{pmatrix},
\end{equation}
where, for any transpose pair $a_k,a_{k+1}$ such that $a_k = a_{k+1} = 0$, the submatrix $\Omega_{e,k}$ of $\Upsilon$ given by (\ref{E: evenOmg}) is substituted with the choice of any nonsingular square matrix of size 2, and
   \begin{equation} \label{E: evenS}
S =    \begin{pmatrix}
-\sqrt{a_1a_2}				&	(r_2-r_1){\rm e}^{\imath t_1}	&		&						&								\\
0						&	\sqrt{a_1a_2}				&		&\hphantom{\sqrt{a_{n-1}a_n}}	&\hphantom{\sqrt{a_{n-1}a_n}}			\\
						&							&\ddots	&						&								\\
\hphantom{\sqrt{a_{n-1}a_n}}	&\hphantom{\sqrt{a_{n-1}a_n}}		&		&-\sqrt{a_{n-1}a_n}			&(r_n-r_{n-1}){\rm e}^{\imath t_{n-1}}	\\
						&							&		&0						&\sqrt{a_{n-1}a_n}		
\end{pmatrix}.
\end{equation}
For odd size $n$ and any choice of nonzero $\omega \in \mathbb{C}$, 
\begin{equation} \label{E: oddUpsilon}
\Upsilon =    \begin{pmatrix}
0		&	0						&	0 										&\vphantom{\sqrt{\frac{a_1}{a_1+a_2}}}		&-\upsilon_{n-1,n}&\upsilon_{n-1,n-1} \, {\rm e}^{\frac{1}{2} \imath (\theta_{n-1} - \theta_n)}		\\
0		&	0						&	0										&\iddots\vphantom{\sqrt{\frac{a_1}{a_1+a_2}}}	&0								&0										\\
0		&	-\upsilon_{2,3}		&\upsilon_{2,2} \, {\rm e}^{\frac{1}{2} \imath (\theta_2 - \theta_3)}	&\vphantom{\sqrt{\frac{a_1}{a_1+a_2}}}		&0								&0										\\
\omega	&	0						&	0										&\hdots\vphantom{\sqrt{\frac{a_1}{a_1+a_2}}}	&\vdots							&\vdots									\\
0		&	\upsilon_{2,2} \, {\rm e}^{\frac{1}{2} \imath (\theta_3 - \theta_2)}		&\upsilon_{2,3}		&\vphantom{\sqrt{\frac{a_1}{a_1+a_2}}}		&0								&0										\\
0		&	0						&	0										&\ddots\vphantom{\sqrt{\frac{a_1}{a_1+a_2}}}	&0								&0										\\
0		&	0						&	0										&\vphantom{\sqrt{\frac{a_1}{a_1+a_2}}}		&\upsilon_{n-1,n-1} \, {\rm e}^{\frac{1}{2} \imath (\theta_n - \theta_{n-1})}	&\upsilon_{n-1,n}	
\end{pmatrix}
\end{equation}
where, for any transpose pair $a_k,a_{k+1}$ such that $a_k = a_{k+1} = 0$, the submatrix $\Omega_{o,k}$ of $\Upsilon$ given by (\ref{E: oddOmg}) is substituted with the choice of any nonsingular square matrix of size 2,
and
   \begin{equation} \label{E: oddS}
S =    \begin{pmatrix}
a_1		& 						&						&		&						&								\\
		& -\sqrt{a_3a_2}			&(r_2-r_3	){\rm e}^{\imath t_2}	&		&\hphantom{\sqrt{a_na_{n-1}}}	&\hphantom{\sqrt{a_na_{n-1}}}			\\
		&	0					&\sqrt{a_3a_2}				&		&						&								\\
		&						&						&\ddots	&						&								\\
		&\hphantom{\sqrt{a_na_{n-1}}}	&\hphantom{\sqrt{a_na_{n-1}}}	&		&-\sqrt{a_na_{n-1}}			&(r_{n-1}-r_n){\rm e}^{\imath t_{n-1}}	\\
		&						&						&		&0						&\sqrt{a_na_{n-1}}		

\end{pmatrix}.
\end{equation}
\end{theorem}

\begin{proof}
Let $n$ be even, and let $Q$ be the permutation-quasidiagonalization of $A$ for even $n$ from Theorem \ref{T:permQuasi}. Because $Q$ is a quasidiagonal matrix consisting of only $2 \times 2$ blocks, $Q$ can be decomposed into a direct sum of $2 \times 2$ matrices $Q_k = \begin{psmallmatrix}
0	&a_{k}	\\
a_{k+1}	& 0
\end{psmallmatrix}$ so that $Q = \bigoplus Q_k^{}$.
Now define $\Omega_e = \bigoplus \Omega_{e,k}$, and
\begin{equation}
\begin{aligned}
S_k &= \Omega_{e,k}^{} Q_k^{} \, \Omega_{e,k}^{-1} 	\\
&= \begin{pmatrix}
-\sqrt{a_{k} a_{k+1}}	& (r_{k+1} - r_{k}	){\rm e}^{\imath t_{k}}		\\
0		 		& \sqrt{a_{k} a_{k+1}}
\end{pmatrix}
\end{aligned}
\end{equation}
so that $S = \bigoplus S_k^{}$.

Since $Q_k$, $\Omega_{e,k}$, and $S_k$ all have the same dimensions, direct sums of these matrices commute with their products.\footnote{Equivalently, the matrix-theoretic way to think of this is that a product of block diagonal matrices is the block-wise product of the blocks.} Thus, 
\[
\begin{aligned}
S &= \hspace{4 pt} \bigoplus S_k^{}	\\
&= \hspace{4 pt} \bigoplus \Omega_{e,k}^{} Q_k^{} \, \Omega_{e,k}^{-1} 			\\
&= (\bigoplus \Omega_{e,k}^{})(\bigoplus Q_k^{})(\bigoplus \Omega_{e,k}^{-1})	\\
&= \Omega_e Q \Omega_e^{-1}	\\
&= \Omega_e P^{-1} A \, P \Omega_e^{-1}	\\
&= (P \Omega_e^{-1})^{-1} A \, (P \Omega_e^{-1})	\\
&= \Upsilon^{-1} A \, \Upsilon
\end{aligned}
\]
where $Q = P^{-1} A P$ follows from the quasidiagonalization of $A$ in Theorem \ref{T:permQuasi} and $\Upsilon = P \, \Omega_e^{-1}$. Therefore, we can conclude $A$ is similar to $S$.

Now, we know by definition, $\Omega_{e,k}$ is $\Gamma$ from Theorem \ref{C:2by2Schur} with $\phi = \frac{1}{2} t_{k}$ and a relabeling of subscripts. Since $\Gamma$ is unitary for all $\phi,t \in \mathbb{R}$, $\Omega_{e,k}$ is unitary for all $t_k \in \mathbb{R}$. Since direct sums of unitary matrices are unitary, $\Omega_e$ is unitary for all $t_k \in \mathbb{R}$ as well, and so must be its inverse $\Omega_e^{-1}$. Finally, $P$ is unitary because it is a permutation matrix, and products of unitary matrices are unitary, so $\Upsilon = P \, \Omega_e^{-1}$ is unitary.

Therefore, since $A = \Upsilon S \, \Upsilon^{-1}$ where $\Upsilon$ is unitary and $S$ is quasidiagonal and upper-triangular, $S$ is a quasidiagonal Schur form for $A$ and $A = \Upsilon S \, \Upsilon^{-1}$ is a Schur decomposition for $A$.

The proof when $n$ is odd is essentially the same where $\Omega_{o,k}$ is used instead of $\Omega_{e,k}$. In this decomposition, $\omega$ from $\Upsilon$ and $\omega^{-1}$ from $\Upsilon^{-1}$ cancel out, leaving $a_1$. The decomposition is independent of the value of $a_1$ and of $\omega$, except $\omega$ must be nonzero since $\Upsilon$ must be nonsingular. Thus, as in Theorem \ref{T:antiEigendec}, the decomposition does not depend on center element $a_1$, but recall $a_1$ is an eigenvalue of $A$, so it does affect the invertibility of $A$. 
\end{proof}

Notice, for even $n$, $k$ only assumes odd values, and for odd $n$, k only assumes even values. For any given $k$, the $l$ in $\upsilon_{k,l}$ only assumes the values $k$ or $k+1$. Also notice $t_k$ is an extra degree of freedom for all $k$ and can be chosen to be any real number, as it does not depend on $A$. A nice particular solution results from setting $t_k = 0$ for all $k$, giving some simplification to $\Upsilon$ and $S$.

Notice also when both elements of each transpose pair have equal modulus, the Schur form of $A$ is equal to the diagonalization of $A$ given by the eigendecomposition in Theorem \ref{T:antiEigendec}. This is equivalent to $A$ being normal, which is in agreement with Theorem \ref{T:unitDiagAntiD}.

It is important to note, if we did not set $\phi = \frac{1}{2} t_{k}$, $\Upsilon$ would still by unitary and $S$ would still be quasidiagonal, but the decomposition would not necessarily be a Schur decomposition because $S$ would no longer necessarily be upper triangular.

A corollary extending the results of Theorem \ref{T:schur} from antidiagonal matrices to unitarily antidiagonalizable matrices follows immediately.

\begin{corollary} [Quasidiagonal Schur Decomposition of a Unitarily Antidiagonalizable Matrix] \label{C:antidiagonalizableSchur}
Every unitarily antidiagonalizable matrix has a quasidiagonal Schur form.
\end{corollary}

\begin{proof}
Let $M$ be a matrix that is unitarily antidiagonalizable to some complex antidiagonal matrix $A$. Since $A$ is has the quasidiagonal Schur form $S$ from Theorem \ref{T:schur}, and since compositions of unitary transformations are unitary, $M$ must have Schur form $S$ as well.
\[
\begin{tikzcd}
M \arrow[d, "unitary"' {xshift = -3 ex, yshift = 1ex}, "similarity"' {xshift = -2ex, yshift = -1 ex}] \arrow[rd,"", "unitary" {xshift = 4.6 ex, yshift = 2.2 ex}, "antidiagonalization" yshift=0.5ex]  &                  \\
S                            & A    \arrow[l, "unitary" yshift=-1.5ex,"similarity" yshift=-3.2ex]  
\end{tikzcd}
\]
\end{proof}

The unitary similarity direct sum decomposition of an antidiagonal matrix also follows from Theorem \ref{T:schur}.

\begin{corollary} [Unitary Similarity Direct Sum Decomposition of an Antidiagonal Matrix] \label{C:unitSimDecomp}
 Let $A$ be a complex antidiagonal matrix of size $n$ with center element $c$ if $n$ is odd. Let $\mathcal{\hat{T}}$ be the set of transpose pairs $\tau = \{\tau_1,\tau_2\}$ in $A$. A direct sum decomposition is given by
\begin{equation} \label{unitDecomp}
A \, \, \unitsim \, \, \mathfrak{S}
\end{equation}
where
\[
\mathfrak{S} = \begin{dcases} 
       \hspace{11 pt} \bigoplus\limits_{\tau \in \mathcal{\hat{T}}} \begin{psmallmatrix} -\sqrt{\tau_1\tau_2} & (\abs{\tau_2}-\abs{\tau_1}){\rm e}^{\imath t_\tau} \\ 0 & \sqrt{\tau_1\tau_2} \end{psmallmatrix}  								& \text{even n} \\
       \smashoperator[r]{\bigoplus\limits_{\tau \in \mathcal{\hat{T}} \setminus \{c,c\}}} \begin{psmallmatrix} -\sqrt{\tau_1\tau_2} & (\abs{\tau_2}-\abs{\tau_1}){\rm e}^{\imath t_\tau} \\ 0 & \sqrt{\tau_1\tau_2} \end{psmallmatrix}  \oplus c \, (1)		& \text{odd n}
       \end{dcases}
\]

for any choice of $t_\tau \in \mathbb{R}$ for all $\tau$, and where $(1)$ is the identity matrix of size 1.
\end{corollary}

\begin{proof}
This is essentially a restatement of Theorem \ref{T:schur}; $S$ from Theorem \ref{T:schur} uniquely determines and is uniquely determined by $\mathfrak{S}$ up to a permutation of the diagonal blocks, where permutations of the diagonal blocks are in one-to-one correspondence with permutations of $\mathcal{\hat{T}}$ treated as an ordered multiset.
\end{proof}

We can also now provide the unitary similarity direct sum decomposition of a unitarily antidiagonalizable matrix.

\begin{corollary} [Unitary Similarity Direct Sum Decomposition of a Unitarily Antidiagonalizable Matrix] \label{C:unitaryDecompAntidiagonalizable}
Every unitarily antidiagonalizable matrix $M$ is unitarily similar to a direct sum decomposition into traceless, upper triangular $2 \times 2$ matrices with an additional $1 \times 1$ matrix if $M$ is of odd size given by (\ref{unitDecomp}). The multiset union, where multiplicities are additive, of the spectra of the direct summands is equal to $spec(M)$.
\end{corollary}

\begin{proof}
This is essentially an abstract algebraic restatement of Corollary \ref{C:antidiagonalizableSchur}
\end{proof}

\section{Acknowledgements} \label{Sect:Acknowledgements}

I am grateful to the intellectual environment provided by \emph{The Symposium: Philosophy Community of Chicago} and for inspiration by Edward Mogul (Loyola University). I would like to express my gratitude to Stephen Walker (University of Chicago), as the idea for this paper can be traced back to a discussion on Buddhist ontology. I thank the staff of \emph{The Violet Hour} for hosting my thoughts on this paper, fueled by various libations. Finally, I express my gratitude for the boundless support and patience of my family, without whom neither this paper, nor I, would exist.

\begin{center}
\emph{He didn't answer. And it's not just that he didn't answer, he didn't know how to answer. --Zhuangzi}
\end{center}


\begin{thebibliography}{99}

   \bibitem{aK11}
      Ali Khajeh-Saeed and J. Blair Perot, \emph{GPU-Supercomputer Acceleration of Pattern Matching}. GPU Computing Gems, 2011.

   \bibitem{aS16}
      A. Sadeghi, \emph{On the Function of Block Anti Diagonal Matrices and Its Applications}. International Journal of Mathematical Modelling and Computations, Vol. 06, No. 02: 105-117, 2016.
   
   \bibitem{wH23}
      Willem H. Haemers and Hatice Topcu, \emph{On signed graphs with at most two eigenvalues unequal to ±1}. Linear Algebra and its Applications, Vol. 670: 68-77, August 2023.
      
   \bibitem{fB18}
      F. Belardo, S.M. Cioab\u a, J.H. Koolen, and J. Wang, \emph{Open problems in the spectral theory of signed graphs}. Art Discr. Appl. Math., 1 \#P2.10, 2018.
      
   \bibitem{eG20}
      E. Ghorbani, W.H. Haemers, H.R. Maimani, and L. Parsaei Majd, \emph{On sign-symmetric signed graphs}. Ars Math. Contemp., 19, pp. 83-93, 2020.

   \bibitem{gG22}
      G.R.W. Greaves, and Z. Stani\' c, \emph{Signed (0, 2)-graphs with few eigenvalues and a symmetric spectrum}. J. Comb. Des., 30, pp. 332-353, 2022.

   \bibitem{wH22}
      W.H. Haemers, and L. Persaei Majd, \emph{Spectral symmetry in conference matrices}. Des. Codes Cryptogr., 90, pp. 1983-1990, 2022.
      
   \bibitem{zS21}
      Z. Stani\' c, \emph{Connected non-complete signed graphs which have symmetric spectrum but are not sign-symmetric}. Examples and Counterexamples, 1, 100007, 2021.

   \bibitem{fR19}
      F. Ramezani, \emph{Some non-sign-symmetric signed graphs with symmetric spectrum}. arXiv:1909.06821, 2019.

   \bibitem{pW23}
      Pepijn Wissing and Edwin R. van Dam, \emph{Spectral fundamentals and characterizations of signed directed graphs}. Journal of Combinatorial Theory, Series A, Volume 187, Article 105573, 2022.
      
   \bibitem{sA18}
      S. Akbaria, H.R. Maimani, and L. Parsaei Majd, \emph{On the spectrum of some signed complete and complete bipartite graphs}. Filomat 32, 5817-5826, 2018. 
      
   \bibitem{eR04}
      E.R. van Dam and E. Spence, \emph{Combinatorial designs with two singular values-I: uniform multiplicative designs}. J. Comb. Theory, Ser. A, 107, pp. 127-142, 2004.
      
   \bibitem{yL22}
      Yuxuan Li, Binzhou Xia, Sanming Zhou, and Wenying Zhu, \emph{A solution to Babai's problem on digraphs with non-diagonalizable adjacency matrix}. arXiv:2208.00887, 2022.
      
   \bibitem{pF69}
      P. Fillmore, \emph{On similarity and the diagonal of a matrix}. Amer. Math. Monthly, 76(2): 167-169, 1969.
      
   \bibitem{aH54}
      A. Horn, \emph{Doubly stochastic matrices and the diagonal of a rotation matrix}. Am. J. Math., 76: 620-630, 1954.
      
   \bibitem{iS23}
      I. Schur, \emph{{\"U}ber eine Klasse von Mittelbildungen mit Anwendungen auf die Determinantentheorie}. Sitzungsber. Berl. Math. Ges. 22, 9-20, 1923.

   \bibitem{zC13}
      Z. B. Charles, M. Farber, C. R. Johnson, and L. Kennedy-Shaffer, \emph{Nonpositive eigenvalues of hollow, symmetric, nonnegative matrices}. SIAM J. Matrix Anal. Appl., 34(3): 1384-1400, 2013.

   \bibitem{mF15}
      M. Farber and C. R. Johnson, \emph{The structure of Schur complements in hollow, symmetric nonnegative matrices with two nonpositive eigenvalues}. Linear Multilinear Algebra, 63(2): 423-438, 2015.
      
   \bibitem{hK16}
      H. Kurata and R. B. Bapat, \emph{Moore-Penrose inverse of a hollow symmetric matrix and a predistance matrix}. Spec. Matrices, 4: 270-282, 2016.
      
   \bibitem{aN18}
      A. Neven and T. Bastin, \emph{The quantum separability problem is a simultaneous hollowisation matrix analysis problem}. J. Phys. A, 51(31), 2018.   
            
   \bibitem{tD20}
      Tobias Damm, Heike Fa\ss bender, \emph{Simultaneous hollowisation, joint numerical range, and stabilization by noise}. January 29, 2020.
      
   \bibitem{lB61}
      L. Brickman, \emph{On the field of values of a matrix}. Proc. Amer. Math. Soc., 12:61-66, 1961.
      
   \bibitem{gL75}
      Glenn R. Luecke, \emph{A Note on Quasidiagonal and Quasitriangular Operators}. Pacific Journal of Mathematics Vol. 56. No. 1. 1975.
      
   \bibitem{dK99}
      D. Kulkarni, D. Schmidt, and S. K. Tsui, \emph{Eigenvalues of tridiagonal pseudo-Toeplitz matrices}. Linear Algebra and its Applications. 297: 63, 1999.

   \bibitem{sN13}
      S. Noschese, L. Pasquini, and L. Reichel, \emph{Tridiagonal Toeplitz matrices: Properties and novel applications}. Numerical Linear Algebra with Applications. 20 (2): 302, 2013.
 
   \bibitem{rM17}
      Ramis Movassagh, Gilbert Strang, Yuta Tsuji, \emph{The Green?s function for the Hückel (tight binding) model}. J. Math. Phys. 58, 033505, 2017. 
 
   \bibitem{jA04}
      J. G. Analytis, S. J. Blundell, and A. Ardavan, \emph{Landau levels, molecular orbitals, and the Hofstadter butterfly in finite systems}. Am. J. Phys. 72, 5, 2004.
 
   \bibitem{hK13}
      H. Karamitaheri, \emph{Thermal and Thermoelectric Properties of Nanostructures}. \url{https://www.iue.tuwien.ac.at/phd/karamitaheri/node16.html}, 2013.
      
   \bibitem{rF00}
      R. N. C. Filho, U. M. S. Costa, and M. G. Cottam, \emph{Green function theory for a magnetic impurity layer in a semi-infinite transverse Ising model}. Journal of Magnetism and Magnetic Materials, 213, 195, 2000.
      
   \bibitem{rD69}
      R. E. De Wames and T. Wolfram, \emph{Theory of Surface Spin Waves in the Heisenberg Ferromagnet}. Phys. Rev. 185, 720, 1969.
            
   \bibitem{mC80}
      M. G. Cottam and D. E. Kontos, \emph{The spin correlation functions of a finite-thickness ferromagnetic slab}. J. Phys. C: Solid State Phys. 13 2945, 1980.
      
   \bibitem{mC76}
      M. G. Cottam, \emph{The spin correlation functions of a semi-infinite Heisenberg ferromagnet}. Journal of Physics C: Solid State Physics 9, 2121, 1976. 

   \bibitem{mA17oct29}
      Maher Ahmed, \emph{Understanding of hopping matrix for 2D materials taking 2D honeycomb and square lattices as study cases}. arxiv.org/pdf/1110.6488.pdf, 2011. 
 
   \bibitem{mA17oct19}
      Maher Ahmed, \emph{Spin Waves in 2D ferromagnetic square lattice stripe}. arxiv.org/pdf/1110.4369v1.pdf, 2011. 

   \bibitem{bN20}
      Branislav K. Nikolic, \emph{How to put magnetic field into tight-binding Hamiltonian}. \url{https://wiki.physics.udel.edu/phys824/How_to_put_magnetic_field_into_tight-binding_Hamiltonian}, 2020.
      
   \bibitem{mA98}
       Max-Albert Knus, Alexander Merkurjev, Markus Rost, and Jean-Pierre Tignol, \emph{The Book of Involutions}. Colloquium Publications, Volume 44, 1998.
       
   \bibitem{pA20}
      Pierre Ablin, \emph{Deep orthogonal linear networks are shallow}. arxiv.org/pdf/2011.13831, 2020.
      
   \bibitem{kH22}
      Koby Hayashi, Sinan G. Aksoy, and Haesun Park, \emph{Skew-Symmetric Adjacency Matrices for Clustering Directed Graphs}. arXiv:2203.01388, 2022.

   \bibitem{kM20}
      Kehelwala Dewage Gayan Maduranga, \emph{Unitary and Symmetric Structure in Deep Neural Networks}. Theses and Dissertations--Mathematics. 77, 2020.
      
   \bibitem{sS20}
      Sanja Singer, \emph{The antitriangular factorization of skew-symmetric matrices}. arxiv.org/pdf/1909.00092, 2020.

   \bibitem{nH18}
      N. L. Harshman, \emph{Symmetry, Structure, and Emergent Subsystems}. January 26, 2018.
      
   \bibitem{hS91}
      Shapiro, Helene, \emph{A Survey of Canonical Forms and Invariants for Unitary Similarity}. Elsevier Science Publishing, 1991.
      
   \bibitem{gDA03}
      G. Donald Allen, \emph{Lectures on Linear Algebra and Matrices}. September 22, 2003.
      
   \bibitem{bS06}
       B. Z. Shavarovskii, \emph{Solvability of Matrix Equations in Rings of Quasi-Diagonal Matrices and Similarity of Matrix Polynomials}. Computational Mathematics and Mathematical Physics; Moscow Vol. 46, Iss. 8, Aug 2006.
       
   \bibitem{rH12}
       Horn, Roger A. and Johnson, Charles R, \emph{Matrix Analysis}. 2nd ed. Cambridge University Press, pp. 33, 2012.
       
   \bibitem{gG96}
       Golub, Gene H. and Van Loan, Charles F., \emph{Matrix Computations}. 3rd ed. The John University University Press, 1996.
       
   \bibitem{bZ62}
       Zumino, Bruno, \emph{Normal Forms of Complex Matrices}. Journal of Mathematical Physics, 3 (5): 1055-1057, 1962.        
       
   \bibitem{dY61}
       Youla, D. C., \emph{A normal form for a matrix under the unitary congruence group}. Can. J. Math. 13: 694-704, 1961.  
       
   \bibitem{jS60}
       J. W. Stander and N. A. Wiegman, \emph{Canonical forms for certain matrices under unitary congruence}. Can. J. Math., 12, 1960.       
 
   \bibitem{zF05}
       Zhang, Fuzhen, \emph{The Schur Complement and its Applications}. Numerical Methods and Algorithms. Vol. 4., 2005.
      
   \bibitem{pP11}
       Philip Powell, \emph{Calculating Determinants of Block Matrices}. arXiv, 2011.
 
   \bibitem{jS17}
       John R Silvester, \emph{Determinants of Block Matrices}. Mathematical Gazette, The Mathematical Association, 84 (501), pp. 460-467, 2000. 
 
   \bibitem{rB70}
      Richard Bronson, \emph{Matrix Methods: An Introduction}. 2nd ed. Academic Press, Chapter 9, 1991.       
       
   \bibitem{mA13}
      Mark A. Armstrong, \emph{Groups and Symmetry}. Springer New York, Chapter 9, 2013.

   \bibitem{mY12}
      Yasuda, Mark, \emph{Some properties of commuting and anti-commuting m-involutions}. Acta Mathematica Scientia, 32 (2): 631-644, 2012.
       
   \bibitem{wYP73}
      W. C. Pye, T. L. Boullion, and T. A. Atchison, \emph{The Pseudoinverse of a Centrosymmetric Matrix}. Linear Algebra and its Applications 6, 201-204, 1973.

\end{thebibliography}
\end{document}